\documentclass{article}
\usepackage[tbtags]{amsmath}
\usepackage{amsfonts,amssymb,mathrsfs,amscd}
\usepackage{graphicx}

\def\thtext#1{
  \catcode`@=11
  \gdef\@thmcountersep{. #1}
  \catcode`@=12
}

\def\threst{
  \catcode`@=11
  \gdef\@thmcountersep{.}
  \catcode`@=12
}

\newtheorem{thm}{Theorem}[section]
\newtheorem{prop}[thm]{Proposition}
\newtheorem{cor}[thm]{Corollary}
\newtheorem{ass}[thm]{Assertion}
\newtheorem{lem}[thm]{Lemma}
\newtheorem{conj}[thm]{Conjecture}
\newtheorem{prb}[thm]{Problem}
\newtheorem{exe}[thm]{Exercise}

\newenvironment{rk}{\trivlist\item[\hskip \labelsep{\bf Remark.}]}{\endtrivlist}

\newenvironment{proof}{\trivlist\item[\hskip \labelsep{\bf Proof.}]}{\endtrivlist}
\newenvironment{examp}{\trivlist\item[\hskip \labelsep{\bf Example.}]}{\endtrivlist}

\newdimen\fgm
 \def\Pic#1#2#3#4{
   \begin{figure}
\begin{center}\includegraphics[height=#4\fgm]{#1}\end{center}
     \caption{\label{#3} #2}
  \end{figure}
 }

 \def\<{\langle}
 \def\>{\rangle}
 \catcode`@=11
 \def\.{.\spacefactor\@m}
 \catcode`@=12

\def\R{\mathbb R}
\def\r{\rho}
\def\g{\gamma}
\def\om{\omega}
\def\G{\Gamma}
\def\D{\Delta}
\def\mst{\operatorname{mst}}
\def\smt{\operatorname{smt}}
\def\mf{\operatorname{mf}}
\def\sr{\operatorname{sr}}
\def\Vor{\operatorname{Vor}}
\def\ch{\operatorname{ch}}
\def\tw{\operatorname{tw}}
\def\cM{{\cal M}}
\def\cG{{\cal G}}
\def\v{\varphi}
\def\:{\colon}
\def\c{\circ}
\def\d{\partial}

\begin{document}

\begin{center}
{\Large \sc Optimal Networks}
\end{center}

\medskip
\begin{center}

{A.\,O.~Ivanov and A.\,A.~Tuzhilin}
\end{center}

\medskip
\begin{quote}
 {\it \qquad
The aim of this mini-course is to give an introduction in Optimal Networks theory. Optimal networks appear as solutions of the following natural problem: How to connect a finite set of points in a metric space in an optimal way? We cover three most natural types of optimal connection: spanning trees\/ {\rm (}connection without additional road forks\/{\rm)}, shortest trees and locally shortest trees, and minimal fillings.
 }
\end{quote}

\bigskip

\section{Introduction: Optimal Connection}
\markright{\thesection.~Introduction.}
This mini-course was given in the First Yaroslavl Summer School on Discrete and Computational Geometry in August~2012, organized by International Delaunay Laboratory ``Discrete and Computational Geometry'' of Demidov Yaroslavl State University. We are very thankful to the organizers for a possibility to give lectures their and to publish this notes, and also for their warm hospitality during the Summer School. The real course consisted of three 1 hour lectures, but the division of these notes into sections is independent on the lectures structure. The video of the lectures can be found in the site of the Laboratory (\verb^http://dcglab.uniyar.ac.ru^). The main reference is our books~\cite{ITBookWP} and~\cite{ITBookRFFI}, and the paper~\cite{ITGromov} for Section~\ref{sec:mf}.

Our subject is optimal connection problems, a very popular and important kind of geometrical optimization problems. We all seek what is better, so optimization problems attract specialists during centuries. Geometrical optimization problems related to investigation of critical points of geometrical functionals, such as length, volume, energy, etc. The main example for us is the length functional, and the corresponding optimization problem consists in finding of length minimal connections.

\subsection{Connecting Two Points}
If we have to points $A$ and $B$ in the Euclidean plane $\R^2$, then, as we know from the elementary school, the shortest curve joining $A$ and $B$ is unique and coincides with the straight segment $AB$, so optimal connection problem is trivial in this case. But if we change the way of distance measuring and consider, for example, so-called Manhattan plane, i.e\. the plane $\R^2$ with fixed standard coordinates $(x,y)$ and the distance function $\r_1(A,B)=|a_1-b_1|+|a_2-b_2|$, where $A=(a_1,a_2)$ and $B=(b_1,b_2)$, then it is not difficult to verify that in this case there are infinitely many shortest curves connecting $A$ and $B$. Namely, if $0\le a_1<b_1$ and $0\le a_2<b_2$, then any monotonic curve $\g(t)=\bigl(x(t),y(t)\bigr)$, $t\in[0,1]$,  $\g(0)=A$, $\g(1)=B$, where functions $x(t)$ and $y(t)$ are monotonic, are the shortest, see~Figure~\ref{fig:manh}, left. Another new effect that can be observed in this example is as follows. In the Euclidean plane a curve such that each its sufficiently small piece is a shortest curve joining its ends (so-called {\em locally shortest curve\/}) is a shortest curve itself. In the Manhattan plane it is not so. The length of a locally shortest curve having the form of the letter $\Pi$, see Figure~\ref{fig:manh}, right,  can be evidently decreased.

\Pic{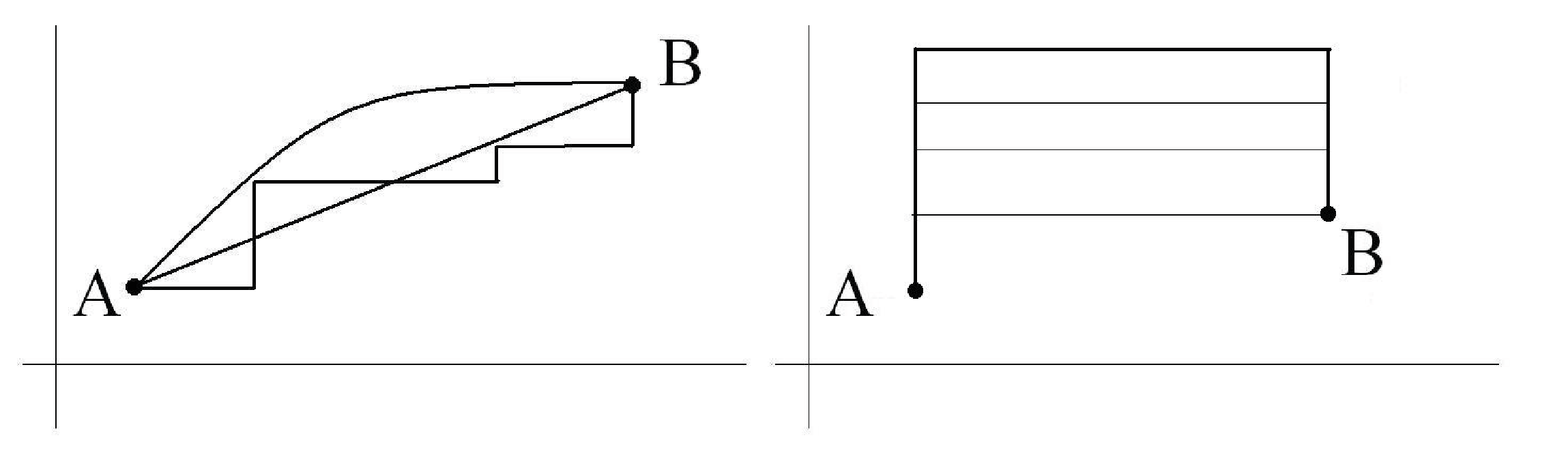}{The shortest curves connecting a pair of points in Manhattan plane (left), and locally shortest but not shortest curve in this plane (right).}{fig:manh}{144}

Similar effects can be observed in the surface of standard sphere $S^2\subset\R^3$. Here the shortest curve joining a pair of points is the lesser arc of the great circle (the cross-section of the sphere by a plane passing through the origin). Two opposite points are connected by infinitely many shortest curves, and if points $A$ and $B$ are not opposite, then the corresponding great circle is unique and it is partitioned into two arcs, both of them are locally shortest, one is the (unique) shortest, but the other one is not. (Really speaking, the difference with the Manhattan plane consists in the fact that for the case of the sphere any directional derivative of the length of any locally shortest arc with respect to its deformation preserving the ends is equal to zero).

\begin{exe}
For a pair of points on the surface of the cube describe shortest and locally shortest curves. Find out an infinite family of locally shortest curves having pairwise distinct lengths.
\end{exe}

\subsection{Connecting Many Points: Possible Approaches}
Let us consider general situation, when we are given with a finite set $M=\{A_1,\ldots,A_n\}$ of points in a metric space $(X,\r)$, and we want to connect them in some optimal way in the sense of the total length of the connection. We are working under assumption that we already know how to connect pairs of points in $(X,\r)$, therefore we need just to organize the set of shortest curves in appropriate way. There are several natural statements of the problem, and we list here the most popular ones.

\subsubsection{No Additional Forks Case: Spanning Trees}
We do not allow additional forks, that is, we can switch between the shortest segments at the points from $M$ only. As a result, we obtain a particular case of Graph Theory problem about minimal spanning trees in a connected weighted graph. We recall only necessary concepts of Graph Theory, the detail can be found, for example in~\cite{Emel}.

Recall that a ({\em simple\/}) {\em graph\/} can be considered as a pair $G=(V,E)$, consisting of a finite set $V=\{v_1,\ldots,v_n\}$ of {\em vertices\/} and a finite set $E=\{e_1,\ldots,e_m\}$ of {\em edges}, where each edge $e_i$ is a two-element subset of $V$. If $e=\{v,v'\}$, then we say that $v$ and $v'$ are {\em neighboring}, edge $e$ {\em joins\/} or {\em connects\/} them, the edge $e$ and each of the vertices $v$ and $v'$ are {\em incident}. The number of vertices neighboring to a vertex $v$ is called the {\em degree of $v$} and is denoted by $\deg v$. A graph $H=(V_H,E_H)$ is said to be a {\em subgraph\/} of a graph $G=(V_G,E_G)$, if $V_H\subset V_G$ and $E_H\subset E_G$. The subgraph $H$ is called {\em spanning}, if $V_H=V_G$.

A {\em path $\g$} in a graph $G$ is a sequence $v_{i_1},e_{i_1},v_{i_2}\ldots,e_{i_k}v_{i_{k+1}}$ of its vertices and edges such that each edge $e_{i_s}$ connects vertices $v_{i_s}$ and $v_{i_{s+1}}$. We also say that the path $\g$ connects the vertices $v_{i_1}$ and $v_{i_{k+1}}$ which are said to be {\em ending vertices\/} of the path. A path is said to be {\em cyclic}, if its ending vertices coincide with each other. A cyclic path with pairwise distinct edges is referred as a {\em simple cycle}. A graph without simple cycles is said to be acyclic. A graph is said to be {\em connected}, if any its two vertices can be connected by a path. An acyclic connected graph is called a {\em tree}.

If we are given with a function $\om\:E\to \R$ on the edge set of a graph $G$, then the pair $(G,\om)$ is referred as a {\em weighted graph}. For any subgraph $H=(V_H,E_H)$ of a weighted graph $\bigr(G=(V_G,E_H),\om\bigl)$ the value $\om(H)=\sum_{e\in E_H}\om(e)$ is called the {\em weight of $H$}. Similarly, for any path $\g=v_{i_1},e_{i_1},v_{i_2}\ldots,e_{i_k}v_{i_{k+1}}$ the value $\om(\g)=\sum_{s=1}^k\om(e_{i_s})$ is called the {\em weight of $\g$}.

For a weighted connected graph $\bigl(G=(V_G,E_G),\om\bigr)$ with positive weight function $\om$, a spanning connected subgraph of minimal possible weight is called {\em minimal spanning tree}. The positivity of $\om$ implies that such subgraph is {\em acyclic}, i.e\. it is a tree indeed. The weight of any minimal spanning tree for $(G,\om)$ is denoted by $\mst(G,\om)$.

Optimal connection problem without additional forks can be considered as minimal spanning tree problem for a special graph. Let  $M=\{A_1,\ldots,A_n\}$ be a finite set of points in a metric space $(X,\r)$ as above. Consider the complete graph $K(M)$ with vertex set $M$ and edge set consisting of all two-element subsets of $M$. In other words, any two vertices $A_i$ and $A_j$ are connected by an edge in $K(M)$. By $A_iA_j$ we denote the corresponding edge. The number of edges in $K(M)$ is, evidently, $n(n-1)/2$. We define the positive weight function $\om_\r(A_iA_j)=\r(A_i,A_j)$. Then any minimal spanning tree $T$ in $K(M)$ can be considered as a set of shortest curves in $(X,\r)$ joining corresponding points and forming a network in $X$ connecting $M$ without additional forks in an optimal way, i.e\. with the least possible length. Such a network is called a {\em minimal spanning tree for $M$ in $(X,\r)$}. Its total weight $\om_\r(T)$ is called {\em length\/} and is denoted by $\mst_X(M)$. In Section~\ref{sec:mst} we speak about minimal spanning trees in more details.

\subsubsection{Shortest tree: Fermat--Steiner Problem}
But already P.~Fermat and C.\,F.~Gauss understood that additional forks can be profitable, i.e\. can give an opportunity to decrease the length of optimal connection. For example, see Figure~\ref{fig:tri}, if we consider the vertex set $M=\{A_1,A_2,A_3\}$ of a regular triangle with side $1$ in the Euclidean plane, then the corresponding graph $K(M)$ consists of three edges of the same weight $1$ and each minimal spanning tree consists of two edges, so $\mst_{\R^2}(M)=2$. But if we add the center $T$ of the triangle and consider the network consisting of three straight segments $A_1T$, $A_2T$, $A_3T$, then its length is equal to $3\frac{2}{3}\frac{\sqrt3}{2}=\sqrt3<2$, so it is shorter than the minimal spanning tree.

\Pic{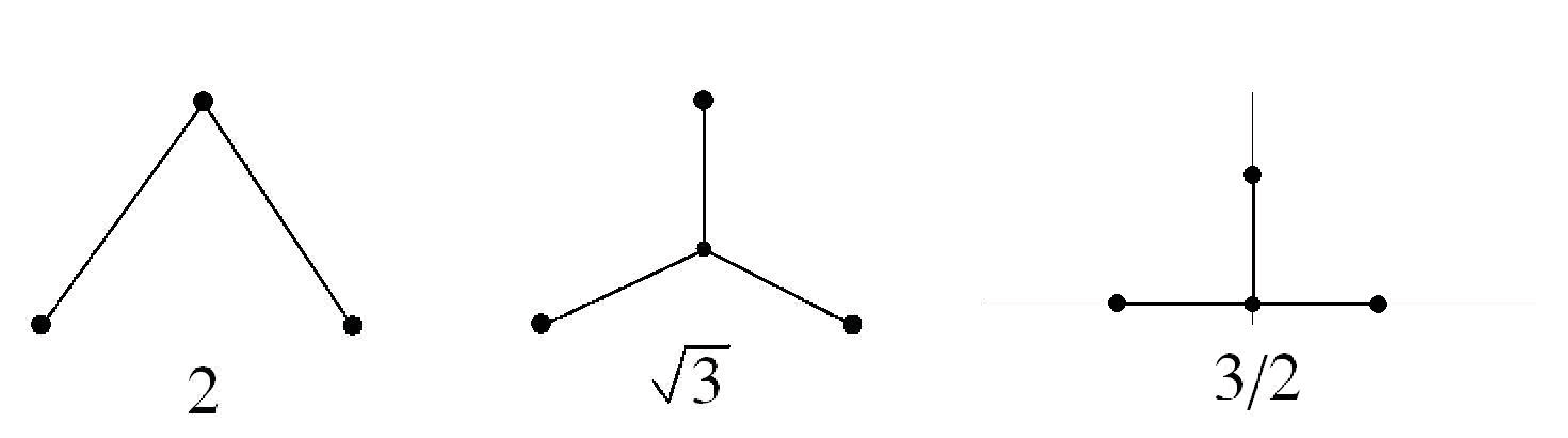}{Minimal spanning tree (left), shortest tree (center), and minimal filling, connecting the vertex set of regular triangle in Euclidean plane.}{fig:tri}{144}

This reasoning leads to the following general definition. Let  $M=\{A_1,\ldots,A_n\}$ be a finite set of points in a metric space $(X,\r)$ as above. Consider a larger finite set $N$, $M\subset N\subset X$, and a minimal spanning tree for $N$ in $X$. Then this tree contains $M$ as a subset of its vertex set $N$, but also may contain some other additional vertices-forks. Such additional vertices are referred  as {\em Steiner points}. Further, we define a value $\smt_X(M)=\inf_{N:M\subset N\subset X}\mst_X(N)$ and call it the {\em length of shortest tree connecting $M$} or of {\em Steiner minimal tree for $M$}. If this infimum attains at some set $N$, then each minimal spanning tree for this $N$ is called a {\em shortest tree\/} or a {\em Steiner minimal tree\/} connecting $M$. Famous Steiner problem is the problem of finding a shortest tree for a given finite subset of a metric space. We will speak about Steiner problem in more details in Section~\ref{sec:smt}. The shortest tree for the vertex set of a regular triangle in the Euclidean plane is depicted in Figure~\ref{fig:tri}.

\subsubsection{Minimizing over Different Ambient Spaces: Minimal Fillings}
Shortest trees give the least possible length of connecting network for a given finite set in a fixed ambient space. But sometimes it s possible to decrease the length of connection by choosing another ambient space. Let  $M=\{A_1,\ldots,A_n\}$ be a finite set of points in a metric space $(X,\r)$ as above, and consider $M$ as a finite metric space with the distance function $\r_M$ obtained as the restriction of the distance function $\r$. Consider an isometric embedding $\v\:(M,\r_M)\to (Y,\r_Y)$ of this finite metric space $(M,\r_M)$ into a (compact) metric space $(Y,\r_Y)$ and consider the value $\smt_Y\bigl(\v(M)\bigr)$. It could be less than $\smt_X(M)$. For example, the vertex set of the regular triangle with side $1$ can be embedded into Manhattan plane as the set $\bigl\{(-1/2,0),(0,1/2),(1/2,0)\bigr\}$, see Figure~\ref{fig:tri}. Than the unique additional vertex of the shortest tree is the origin and the length of the tree is $3/2<\sqrt3$.
So, for a finite metric space $\cM=(M,\r_M)$, consider the value $\mf(\cM)=\inf_\v\smt_Y\bigl(\v(M)\bigr)$ which is referred as {\em weight of minimal filling\/} of the finite metric space $\cM$. Minimal fillings can be naturally defined in terms of weighted graphs and can be considered as a generalization of Gromov's concept of minimal fillings for Riemannian manifolds. We speak about them in more details in Section~\ref{sec:mf}.

\section{Minimal Spanning Trees}
\markright{\thesection.~Minimal Spanning Trees}\label{sec:mst}
In this section we discuss minimal spanning trees construction in more details. As we have already mentioned above, in this case the problem can be stated in terms of Graph Theory for an arbitrary connected weighted graph. But geometrical interpretation permits to speed up the algorithms of Graph Theory.

\subsection{General Case: Graph Theory Approach}
We start with the Graph Theory problem of finding a minimal spanning tree in a connected weighted graph. It is not difficult to verify that direct enumeration of all possible spanning subtrees of a connected graph leads to an exponential algorithm.

To see that, recall well-known Kirchhoff theorem counting the number of spanning subtrees. If $G=(V,E)$ is a connected graph with enumerated vertex set $V=\{v_1,\ldots,v_n\}$, then its {\em Kirchhoff matrix\/} is defined as $(n\times n)$-matrix $B_G=(b_{ij})$ with elements
$$
b_{ij}=\left[\begin{array}{ll} \deg v_i & \quad\text{if $i=j$,}\\
                              -1 &\quad \text{if $\{v_i,v_j\}\in E$,}\\
                              0 & \quad \text{otherwise.}
\end{array}\right.
$$
Then the following result based on elementary Graph Theory and Binet--Cauchy formula for determinant calculation is valid, see proof, for example in~\cite{Emel}.

\begin{thm}[Kirchhoff]
For a connected graph $G$ with $n\ge2$ vertices, the number of spanning subtrees is equal to the algebraic complement of any element of the Kirchhoff matrix $B_G$.
\end{thm}

\begin{examp}
Let $G=K_n$ be the complete graph with $n$ vertices. Than its Kirchhoff matrix has the following form:
 $$
B_{K_n}=\left(\begin{array}{ccccc} n-1 & -1 & -1& \cdots & -1 \\
                                   -1 & n-1&  -1& \cdots & -1 \\
                                   \vdots& \vdots&\vdots&\ddots&\vdots \\
                                   -1&-1&-1&\cdots&n-1
                                     \end{array}\right).
 $$
\end{examp}
The algebraic complement of the element $b_{nn}$ is equal to
 $$
\left|\begin{array}{cccc} n-1 & -1 & \cdots & -1 \\
                            -1 & n-1& \cdots & -1 \\
                         \vdots& \vdots&\ddots&\vdots \\
                              -1&-1&\cdots&n-1
                                     \end{array}\right|=
\left|\begin{array}{ccc} 1 & \cdots & 1 \\
                            -1 & \cdots & -1 \\
                         \vdots& \ddots&\vdots \\
                              -1&\cdots&n-1
                                     \end{array}\right|=
\left|\begin{array}{cccc} 1 & 1& \cdots & 1 \\
                          0 & n& \cdots & 0 \\
                         \vdots& \vdots&\ddots&\vdots \\
                          0 & 0&\cdots & n
                                     \end{array}\right|=n^{n-2},
 $$
where the first equality is obtained by change of the first row by the sum of all the rows, and the second equality is obtained by change of the $i$th row, $i\ge2$, by the sum of it with the first row.

\begin{cor}
The complete graph with $n$ vertices contains $n^{n-2}$ spanning trees.
\end{cor}

\begin{rk}
Notice that this result is equivalent to Cayley Theorem saying that the total number of trees with $n$ enumerated vertices is equal to $n^{n-2}$.
\end{rk}

But it is a surprising fact, that there exist polynomial algorithms constructing minimal spanning trees. Several such algorithms were discovered in 1960s. We tell about Kruskal's algorithm. Similar Prim's algorithm can be found in~\cite{Emel}.

So, we are given with a connected weighted graph $\bigl(G=(V,E),\om\bigr)$ with positive weight function $\om$.
At the initial step of Kruckal algorithm we construct the graph $T_0=(V,\emptyset)$ and put $E_0=E$. If the graph $T_{i-1}$ and the non-empty set $E_{i-1}\subset E$, $i\le n-1$, have been already constructed, then we choose in $E_{i-1}$ an edge $e_i$ of least possible weight and construct a new graph $T_i=T_{i-1}\cup e_i$ and also a new set $E_i=\{e\in E\mid \text{$e\not\in T_i$ and $T_i\cup e$ is acyclic}\}$. Algorithm stops when the graph $T_{n-1}$ is constructed\footnote{
Here the operation of adding an edge $e$ to a graph $G=(V,E)$ can be formally defined as follows: $G\cup e=\left(V,E\cup\{e\}\right)$. Similarly, $G\setminus e=\left(V,E\setminus\{e\}\right)$.
}.

\begin{thm}[Kruskal]
Under the above notations, the graph $T_{n-1}$ can be constructed for any connected weighted graph $\cG=(G,\om)$, and moreover, $T_{n-1}$
is a minimal spanning tree in $\cG$.
\end{thm}

\begin{proof}
The set $E_i$ is non-empty for all $0\le i\le n-2$, because the corresponding subgraphs $T_i$ are not connected (the graph $T_i$ has $n$ vertices and $i$ edges), therefore all the graphs $T_1,\ldots,T_{n-1}$ can be constructed. Further, all these graphs are acyclic due to the construction, and $T_{n-1}$ has $n$ vertices and $n-1$ edges, so it is a tree.

To finish the proof it remains to show that the spanning tree $T_{n-1}\subset G$ is minimal. Since the graph $\cG$  has a finite number of spanning trees, a minimal spanning tree does exist. Let $T$ be a minimal spanning tree. We show that it can be reconstructed to the tree $T_{n-1}$ without changing the total weight, so $T_{n-1}$ is also a minimal spanning tree.

To do this, recall that the edges of the tree $T_{n-1}$ are enumerated in accordance with the work of the algorithm. Denote them by $e_1,\ldots,e_{n-1}$ as above, and assume that $e_k$ is the first one that does not belong to $T$. The graph $T\cup e_k$ contains a unique cycle $c\supset e_k$. This cycle $c$ also contains an edge $e$ not belonging to $T_{n-1}$ (otherwise $c\subset T_{n-1}$, a contradiction). Consider the graph $T'=T\cup e_k \setminus e$. It is evidently a spanning tree in $G$, and therefore its weight is not less than the weight of the minimal spanning tree $T$, hence
$$
\om(T')=\om(T)+\om(e_k)-\om(e)\ge \om(T),
$$
and thus, $\om(e_k)\ge\om(e)$.

On the other hand, all the edges $e_1,\ldots, e_{k-1}$ belongs to $T$ by our assumption. Therefore, the graph $T_{k-1}\cup e$ is a subgraph of $T$ and is acyclic, in particular. Hence, $e\in E_{k-1}$ so as $e_k\in E_{k-1}$. But the algorithm has chosen $e_k$, hence $\om(e_k)\le\om(e)$. Thus, $\om(e_k)=\om(e)$, and so $\om(T')=\om(T)$, and therefore $T'$ is a minimal spanning tree in $\cG$. But now $T'$ contains the edges $e_1,\ldots,e_k$ from $T_{n-1}$. Repeating this procedure we reconstruct $T$ to $T_{n-1}$ in the class of minimal spanning trees. Theorem is proved.
\end{proof}

\begin{rk}
For a connected weighted graph with $n$ vertices and $m$ edges the complexity of the Kruskal's algorithm can be naturally estimated as $mn\sim n^3$. The estimation can be improved to $m \log m\sim n^2\log n$. The fastest non-randomized comparison-based algorithm with known complexity belongs to Bernard Chazelle~\cite{Chaz}. It turns out that if the weight function is geometrical, then the algorithms can be improved.
\end{rk}

\subsection{Euclidean Case: Geometrical Approach}
Now assume that $M$ is a finite subset of the Euclidean plane $\R^2$. It turns out that a minimal spanning tree for $M$ in $\R^2$ can be constructed faster than the one for an abstract complete graph with $n=|M|$ vertices by means of  some geometrical reasonings. To do that we need to construct so called Voronoi partition of the plane, corresponding to $M$, and the Delaunay graph on $M$. It turns out that any minimal spanning tree for $M$ in $\R^2$ is a subgraph of the Delaunay graph, see Figure~\ref{fig:vordel}, and the number of edges in this graph is linear with respect to $n$, so the standard Kruskal's algorithm applied to it gives the complexity $n\log n$ instead of $n^2 \log n$ for the complete graph with $n$ vertices.

Let us pass to details. Let $M=\{A_1,\ldots, A_n\}\subset\R^2$ be a finite subset of the plain. The {\em Voronoi cell\/} of the point $A_i$ is defined as
$$
\Vor_M(A_i)=\left\{x\in\R^2\mid \text{$\|x-A_i\|\le\|x-A_j\|$ for all $j$}\right\}.
$$
The Voronoi cell for $A_i$ is a convex polygonal domain which is equal to the intersection of the closed half-planes restricted by the perpendicular bisectors of the segments $A_iA_j$, $j\ne i$. It is easy to verify, that the intersection of any two Voronoi cells has no interior points and that $\cup_i\Vor_M(A_i)=\R^2$. This partition of the plane is referred as {\em Voronoi partition\/} or {\em Voronoi diagram}. Two cells $\Vor_M(A_i)$ and $\Vor_M(A_j)$ are said to be {\em adjacent}, if there intersection contains a straight segment.
The {\em Delaunay graph $D(M)$} is defined as the dual planar graph to the Voronoi diagram. More precisely, the vertex set of $D(M)$ is $M$, and to vertices $A_i$ and $A_j$ are connected by an edge, if and only if their Voronoi cells $\Vor_M(A_i)$ and $\Vor_M(A_j)$ are adjacent. The edges of the Delaunay graph are the corresponding straight segments.

\Pic{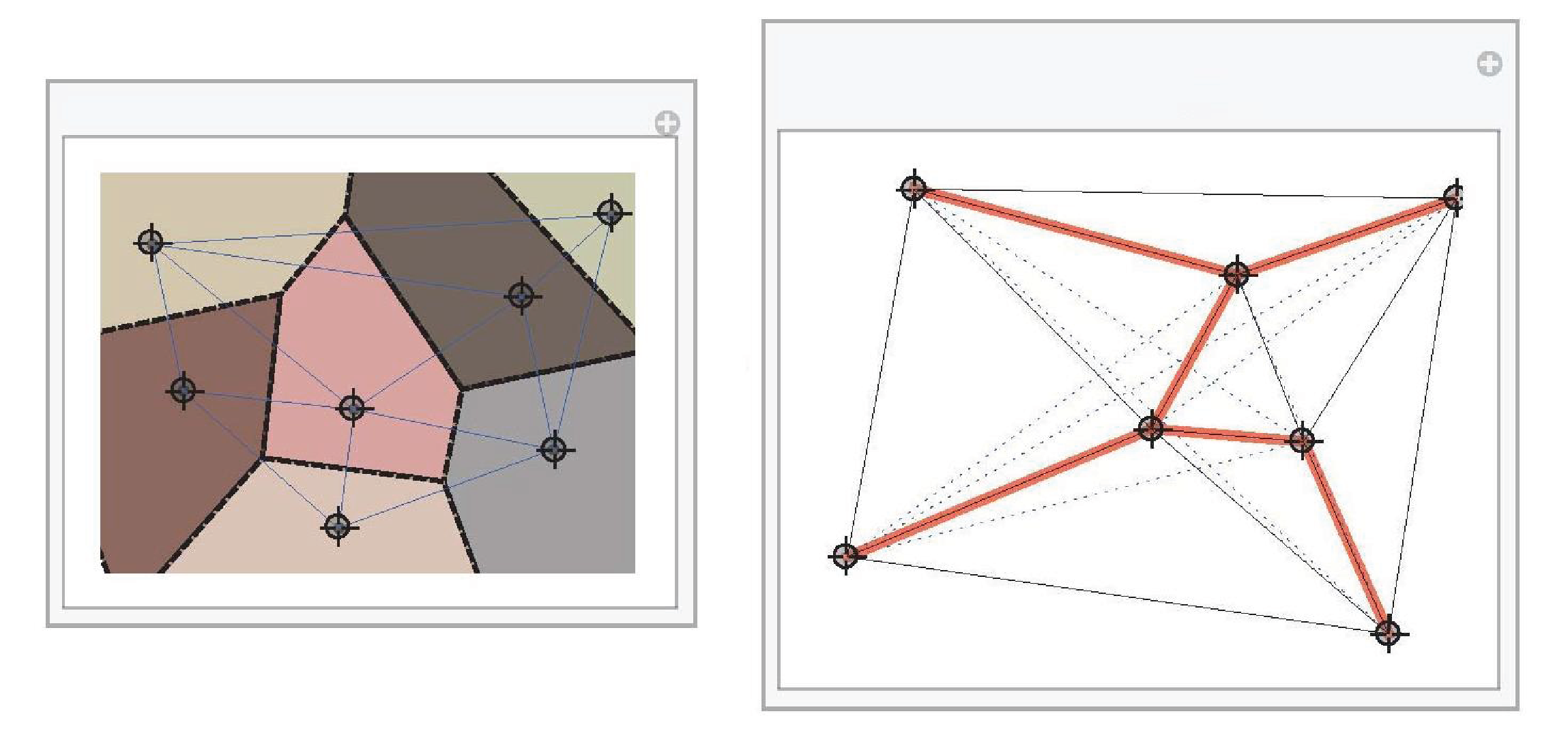}{Voronoi diagram (left) and Delaunay graph together with a minimal spanning tree (right) for a point set in the plane.}{fig:vordel}{196}

It is easy to verify, that if the set $M$ is generic in the sense that no three points lie at a common straight line and no four points lie at a common circle, then the Delaunay graph $D(M)$ is a triangulation, i.e. its bounded faces are triangles. In general case some bounded faces could be inscribed polygons. Anyway, the number of edges of the graph $D(M)$ does not exceed $3n$. It remains to prove the following key Lemma.

\begin{lem}
Any minimal spanning tree for $M\subset\R^2$ is a subgraph of the Delaunay graph $D(M)$.
\end{lem}

\begin{proof}
Let $e=A_iA_j$ be an edge of a minimal spanning tree $T$ for $M$. We have to show that the Voronoi cells $\Vor_M(A_i)$ and $\Vor_M(A_j)$ are adjacent. The graph $T\setminus e$ consists of two connected components, and this partition generates a partition of the set $M$ into two subsets, say $M_1$ and $M_2$. Assume that $A_i\in M_1$ and $A_j\in M_2$. The minimality of the spanning tree $T$ implies that  $\|A_i,A_j\|$ is equal to the distance between the sets $M_1$ and $M_2$, where $\|A_i,A_j\|$ stands for the distance between $A_i$ and $A_k$.

By $u$ we denote the middle point of the straight segment $A_iA_j$, and let $A_k$ be another point from $M$. Assume that $A_k\in M_2$. Due to the previous remark, $\|A_i,A_k\|\ge\|A_i,A_j\|$, therefore
$$
\|u,A_k\|\ge\|A_i,A_k\|-\|A_i,u\|\ge \|A_iA_j\|-\|A_i,u\|=\|A_i,A_j\|/2=\|u,A_i\|=\|u,A_j\|.
$$
On the other hand, if $\|u,A_k\|=\|u,A_i\|$, then we have equalities in both above inequalities. The first one means that  $u$ lies at the straight segment $A_iA_k$, hence $A_k$ lies at the ray $A_iu$. The second equality implies $\|A_iA_k\|=\|A_iA_j\|$, and so $A_k=A_j$, a contradiction. Thus, $\|u,A_k\|>\|u,A_i\|$, that is $u$ does not belong to the cell $\Vor_M(A_k)$ for $k\not\in\{i,\,j\}$.
Thus, $u\in\Vor_M(A_i)\cap\Vor_M(A_j)$,

Since the inequality proved is strict, the same arguments remain valid for points lying close to $u$ on the perpendicular bisector to the segment $A_iA_j$. Therefore, the intersection of the Voronoi cells $\Vor_M(A_i)$ and $\Vor_M(A_j)$ contains a straight segment, that is the cells are adjacent. Lemma is proved.
\end{proof}

\begin{rk}
The previous arguments work in any dimension. But the trouble is that starting from the dimension $3$ the number of edges in the Delaunay graph need not be linear on the number of its vertices.
\end{rk}

\begin{exe}
Verify that the same arguments can be applied to minimal spanning trees for a finite subset of $\R^n$.
\end{exe}

\begin{exe}
Give an example of a finite subset $M\subset\R^3$ such that the Delaunay graph $D(M)$ coincides with the complete graph $K(M)$.
\end{exe}

\begin{prb}
In what metric spaces similar geometrical approach also works? It definitely works for planar polygons with intrinsic metric, see~\cite{ITInTrees}.
\end{prb}

\section{Steiner Trees and Locally Minimal Networks}\label{sec:smt}
\markright{\thesection.~Steiner Trees and Locally Minimal Networks.}
In this section we speak about shortest trees and locally shortest networks in more details. Besides necessary definitions we discuss local structure theorems, Melzak algorithm constructing locally minimal trees in the plane, global results concerning locally minimal binary trees in the plane (so called twisting number theory) and the particular case, locally minimal binary trees with convex boundaries (language of triangular tilings). The details concerning twisting number and tiling realization theory can be found in~\cite{ITBookRFFI} or~\cite{ITBookWP}, and also in~\cite{ITPlane}.

\subsection{Fermat Problem}
The idea that additional forks can help to decrease the length of a connecting network had been already clear to P.~Fermat and his students. It seems that Fermat was the first, who stated the following optimization problem: for given three points $A_1$, $A_2$, and $A_3$ in the plane find a point $X$ minimizing the sum of distances from the points $A_i$, i.e\. minimize the function $F(X)=\sum_i\|A_i,X\|$. For the case when all the angles of the triangle $A_1A_2A_3$ are less than or equal to $120^\c$ the solution was found by E.~Torricelli and later by R.~Simpson. The construction of Torricelli is as follows, see Figure~\ref{fig:Torr}.

\Pic{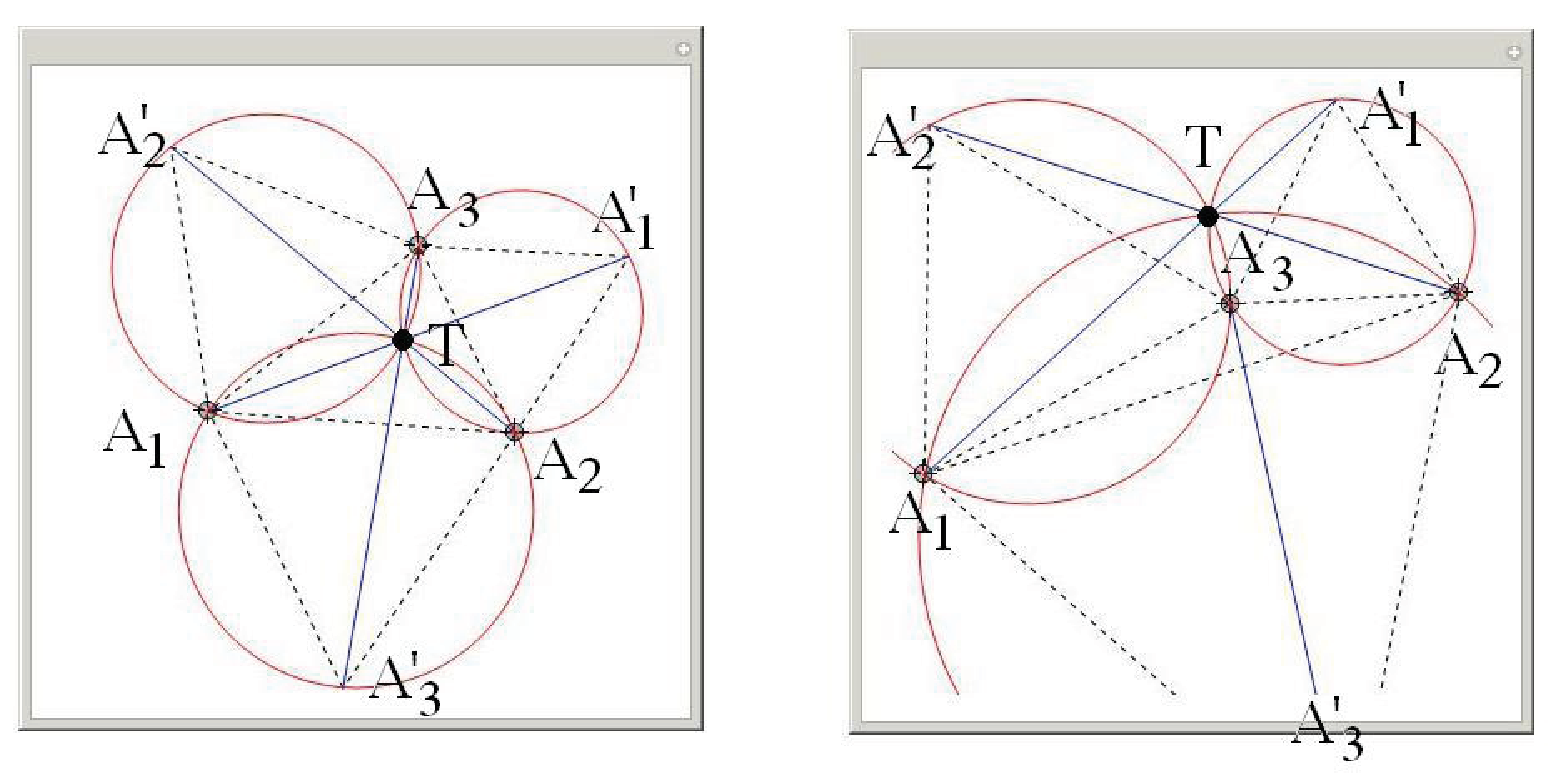}{Torricelli--Simpson construction, the case $A_i\le120^\c$ (left), and the case $A_3>120^\c$ (right).}{fig:Torr}{196}

On the sides of the triangle $A_1A_2A_3$ construct equilateral triangles $A_iA_jA'_k$, $\{i,j,k\}=\{1,2,3\}$, such that they intersect the initial triangle only by the common sides. Then, as Torricelli proved, the circumscribing circles of these three triangles intersect in a point referred as {\em Torricelli point $T$ of the triangle $A_1A_2A_3$}. If all the angles $A_i$ are less than or equal to $120^\c$, then $T$ lies in the triangle $A_1A_2A_3$ and gives the unique solution to the Fermat problem.\footnote{
An elementary proof can be obtained by rotation $R$ of a copy of the triangle around its vertex, say $A_1$, by $60^\c$ and considering the polygonal line $L$ joining $A_2$, $X$, image $R(X)$ of $X$ under the rotation, and $R(A_3)$. The length of $L$ is equal to $F(X)$, and minimal value of $F(X)$ corresponds to the location of the $X$ such that $L$ is a straight segment.
}
Later Simpson proved that the straight segments $A_iA'_i$ also pass through the Torricelli point, and the lengths of all these three segments are equal to $F(T)$. If one of the angles, say $A_3$, is more than $120^\c$, then the Torricelli point is located outside the triangle and can not be the solution to Fermat problem. In this case the solution is $X=A_3$.

\begin{rk}
So we see, that shortest tree for a triangle in the plane consists of straight segments meeting at the vertices by angles more than or equal to $120^\c$. It turns out, that this {\em $120^\c$-property\/} remains valid in much more general situation.
\end{rk}

\subsection{Local Structure Theorem and Locally Minimal Networks}
Let $M=\{A_1,\ldots,A_n\}$ be a finite subset of Euclidean space $\R^N$, and $T$ is a Steiner tree connecting $M$. Recall that we defined shortest trees as abstract graphs with vertex set in the ambient metric space. In the case of $\R^N$ it is natural to model edges of such graph as straight segments joining corresponding points in the space. The configuration obtained is referred as a {\em geometrical  realization of the corresponding graph}. Below, speaking about shortest trees in $\R^N$ we will usually mean their geometrical realizations.
The local structure of a shortest tree (more exactly of a geometrical realization of the tree) can be easily described.

\begin{thm}[Local Structure]\label{th:LStr}
Let $\G$ be a shortest tree connecting a finite subset $M=\{A_1,\ldots,A_n\}$ in $\R^N$. Then
\begin{enumerate}
 \item all edges of $\G$ are straight segments{\em ;}
 \item any vertex $v\in\G$ of degree $1$ belongs to $M${\em ;}
 \item any two neighboring edges of $\G$ meet in common vertex by angle more than or equal to $120^\c${\em ;}
 \item if the degree of a vertex $v$ is equal to $2$ and $v\not\in M$, then the edges meet at $v$ by $180^\c$ angle.
\end{enumerate}
\end{thm}

\begin{cor}
Let $\G$ be a shortest tree connecting a finite subset $M=\{A_1,\ldots,A_n\}$ in $\R^N$. Then the degree of any its vertex is at most $3$, and if the degree of a vertex $v$ equals to $3$, then the edges meet at $v$ by angles equal to $120^\c$.
\end{cor}

\begin{examp}
Let $M$  be the vertex set of regular tetrahedron $\D$ in $\R^3$. Then the network consisting of four straight segments joining the vertices of the tetrahedra with its center $O$ is not a shortest network. Indeed, since $\deg O=4$, then the angles between the edges meeting at $O$ are less than $120^\c$. The set $M$ is connected by three different (but isometrical) shortest networks, each of which has two additional vertices of degree $3$, see Figure~\ref{fig:tetrah}.

\Pic{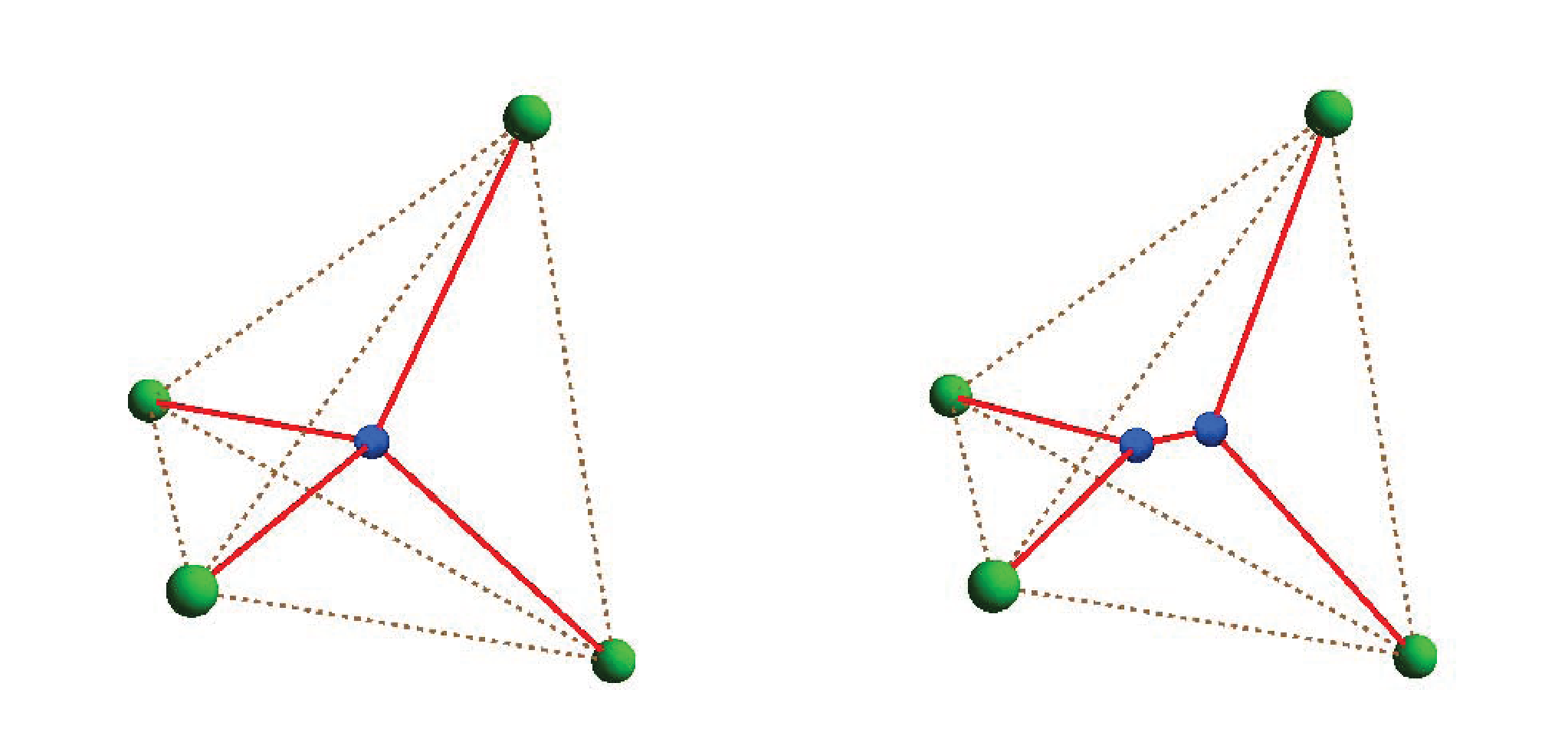}{Non-shortest tree (left) and one of the shortest trees (right) for the vertex set of a regular tetrahedron.}{fig:tetrah}{196}

\end{examp}

Theorem~\ref{th:LStr} can be just ``word-by-word'' extended to the case of Riemannian manifolds (we only need to change straight segments by geodesic segments)~\cite{ITBookWP} and even to the case of Alexandrov spaces with bounded curvature. The case of normed spaces turned out to be more complicated (some general results can be found in~\cite{ITBookRFFI}).

A connected graph $\G$ in $\R^N$ (in a Riemannian manifold) whose vertex set contains a finite subset $M\subset\R^N$ is called a {\em locally minimal network connecting $M$} or {\em with the boundary $\d\G=M$}, if it satisfies Conditions (1)--(4) from Theorem~\ref{th:LStr}. In the case of complete Riemannian manifolds such graphs are minimal ``in small,'' i.e\. the following result holds, see~\cite{ITBookWP}.

\begin{thm}[Minimality ``in small'']\label{th:small}
Let $\G$ be a locally minimal network connecting a finite subset $M$ of a complete Riemannian manifold $W$. Then each point $P\in \G$ possesses a neighborhood $U$ in $W$, such that the network $\G\cap U$ is a shortest network with the boundary $(\d\G\cap U)\cup(\G\cap\d U)$. \end{thm}

\begin{rk}
In the case of normed spaces Theorem~\ref{th:small} is not valid even for two-point sets, see example in Figure~\ref{fig:manh}.
\end{rk}

\subsection{Melzak Algorithm and Steiner Problem Complexity}
Let us return back to the case of Euclidean plane. It turns out that in this case the Torricelli--Simpson construction can be generalized to a geometrical algorithm, that either constructs a locally minimal tree of a given structure for a given boundary set, or reports that such a tree does not exist. This algorithm was discovered by Z.~Melzak~\cite{Melz} and improved by F.~Hwang~\cite{Hw}.

Assume that we are given with a tree $G$ whose vertex degrees are at most $3$, a finite subset $M$ of the plane, and a bijection $\v\:\d G\to M$, where $\d G$ is the set of all vertices from $G$ of degrees $1$ and $2$. To start with, partition the tree $G$ into the union of so-called {\em non-degenerate\/} components $G_i$ by cutting the tree at each its vertex of degree $2$, see Figure~\ref{fig:compon}. To construct locally minimal network $\G$ of type $G$ spanning $M$ in accordance with $\v$ it suffices to construct each its component $\G_i$ of type $G_i$ on the corresponding boundary $M_i=\v(\d G_i)$, where $\d G_i=\d G\cap G_i$, in accordance with $\v_i=\v|_{\d G_i}$ and to verify the angles between the edges of the components at the vertices of degree $2$. All these angles must be more than or equal to $120^\c$, see~Figure~\ref{fig:compon}.

\Pic{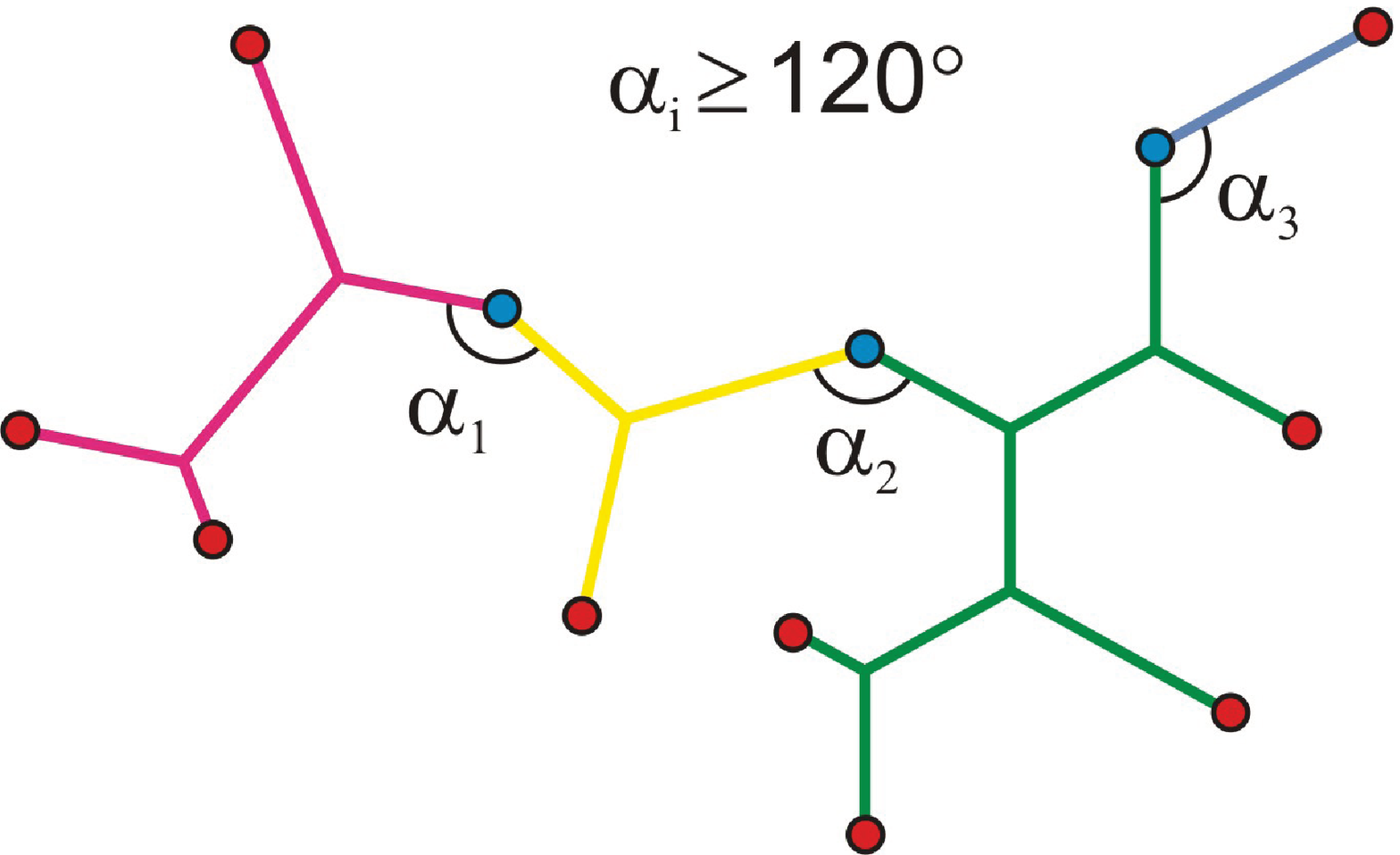}{Tree $G$ is partitioned into $4$ non-degenerate components by cutting at vertices of degree $2$.}{fig:compon}{144}

Now we pass to the case of one non-degenerate component, i.e\. we assume that $G$ has no vertices of degree $2$ and that $\d G$ consists of all the vertices of degree $1$. Such trees are referred as {\em binary}. If $|\d G|=2$, then the corresponding locally minimal tree $\G$ is a straight segment. Otherwise, it is easy to verify that each such tree $G$ contains so-called {\em moustaches}, i.e\. a pair of vertices of degree $1$ neighboring with a common vertex of degree $3$. Fix such moustaches $\{x, x'\}\subset\d G$, by $y$ denote their common vertex of degree $3$, and make {\em a forward step of Melzak algorithm}, see Figure~\ref{fig:Mel1}, that reduces the number of boundary vertices by $1$. Namely, we reconstruct the tree $G$ by deleting the vertices $x$ and $x'$ together with the edges $xy$ and $x'y$ and adding $y$ to the boundary of new binary tree; reconstruct the set $M$ by deleting the points $\v(x)$ and $\v(x')$ and adding a new point $A_{xx'}$ which is the third vertex of a regular triangle constructed on the straight segment $\v(x)\v(x')$ in the plane; and reconstruct the mapping $\v$ putting $\v(y)=A_{xx'}$. Notice that the point $A_{xx'}$ can be constructed in two ways, because there are two such regular triangles. Thus, if the number of boundary vertices in the resulting tree is more than $2$, then we can repeat the procedure described above. And if it becomes $2$, then we can construct the corresponding locally minimal tree --- the straight segment. Here the forward trace of Melzak algorithm stops. Now we have to reconstruct the initial tree, if possible.

\Pic{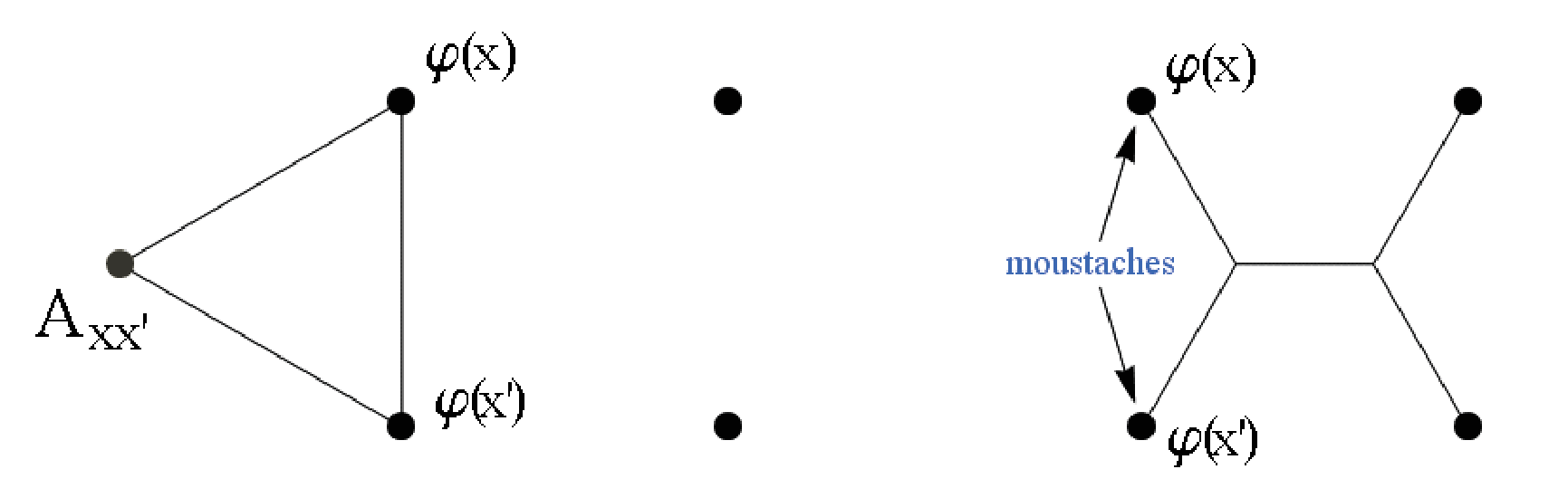}{A step of forward trace of Melzak algorithm.}{fig:Mel1}{144}

Thus, we have a straight segment $I\subset\R^2$ realizing locally minimal tree with unique edge $ab$, and at least one of its ending points has the form $A_{xx'}$, where $x$ and $x'$ are the boundary vertices of the binary tree $G$ from the previous step, neighboring with their common vertex of degree $3$. Let this common vertex be $a$, that is $a$ corresponds to $A_{xx'}$. We reconstruct $G$ by adding edges $ax$ and $ax'$. Then we restore the points $\v(x)$ and $\v(x')$ in the plane together with the regular triangle $\v(x)\v(x')A_{xx'}$, circumscribe the circle $S^1$ around it and consider the intersection of $S^1$ with the segment $I$, see Figure~\ref{fig:mel2}. If it does not contains a point lying at the smaller arc of $S^1$ restricted by $\v(x)$ and $\v(x')$, then the tree $G$ can not be reconstructed and we have to pass to another realization of the forward trace of the algorithm. Otherwise we put $\v(a)$ be equal to this point. The straight segments $\v(x)\v(a)$ and $\v(x')\v(a)$ meet at $\v(a)$ by $120^\c$ and together with the subsegment $\v(a)\v(b)$ form  a locally minimal binary tree $\G$ of type $G$ with tree boundary vertices. We repeat this procedure until we either reconstruct the tree of type $G$, or verify all possible realizations of the forward trace and conclude that the tree of type $G$ does not exists.

\Pic{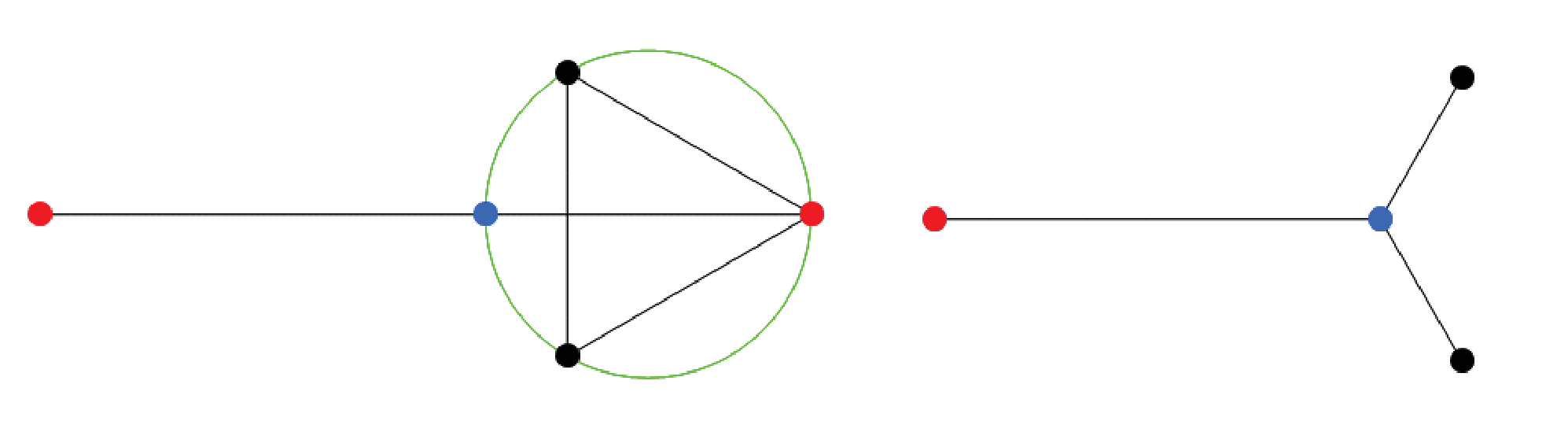}{A step of back trace of Melzak algorithm.}{fig:mel2}{144}

The Melzak algorithm described above contains an exponential number of possibilities of its forward trace realization, due to two possible locations of each regular triangle constructed by the algorithm. This complexity can be reduced by means of modification suggested by F.~Hwang~\cite{Hw}. He showed that considering a bit more complicated configurations of boundary points (four points corresponding to ``neighboring moustaches'' or three points corresponding to moustaches and ``neighboring'' degree-$1$ vertex) one always can understand which regular triangle must be chosen, see details in~\cite{Hw}.

But unfortunately even a linear time realization of Melzak algorithm does not lead to a polynomial algorithm of a shortest tree finding. The reason is a huge number of possible structures of the tree $G$ with $|\d G|=n$ together with also exponential number of different mappings $\v\:\d G\to M$ for fixed $\d G$ and $M$. Even for binary trees we have $3$ possibilities for $n=4$, see Figure~\ref{fig:manytop}, and $15$ possibilities for $n=5$ (notice that the corresponding binary trees are isomorphic as graphs). For $n=6$ we have two non-isomorphic binary trees and the number of possibilities becomes $90$. It can be shown that the total number of possibilities can be estimated by Catalan number and grough exponentially.

\Pic{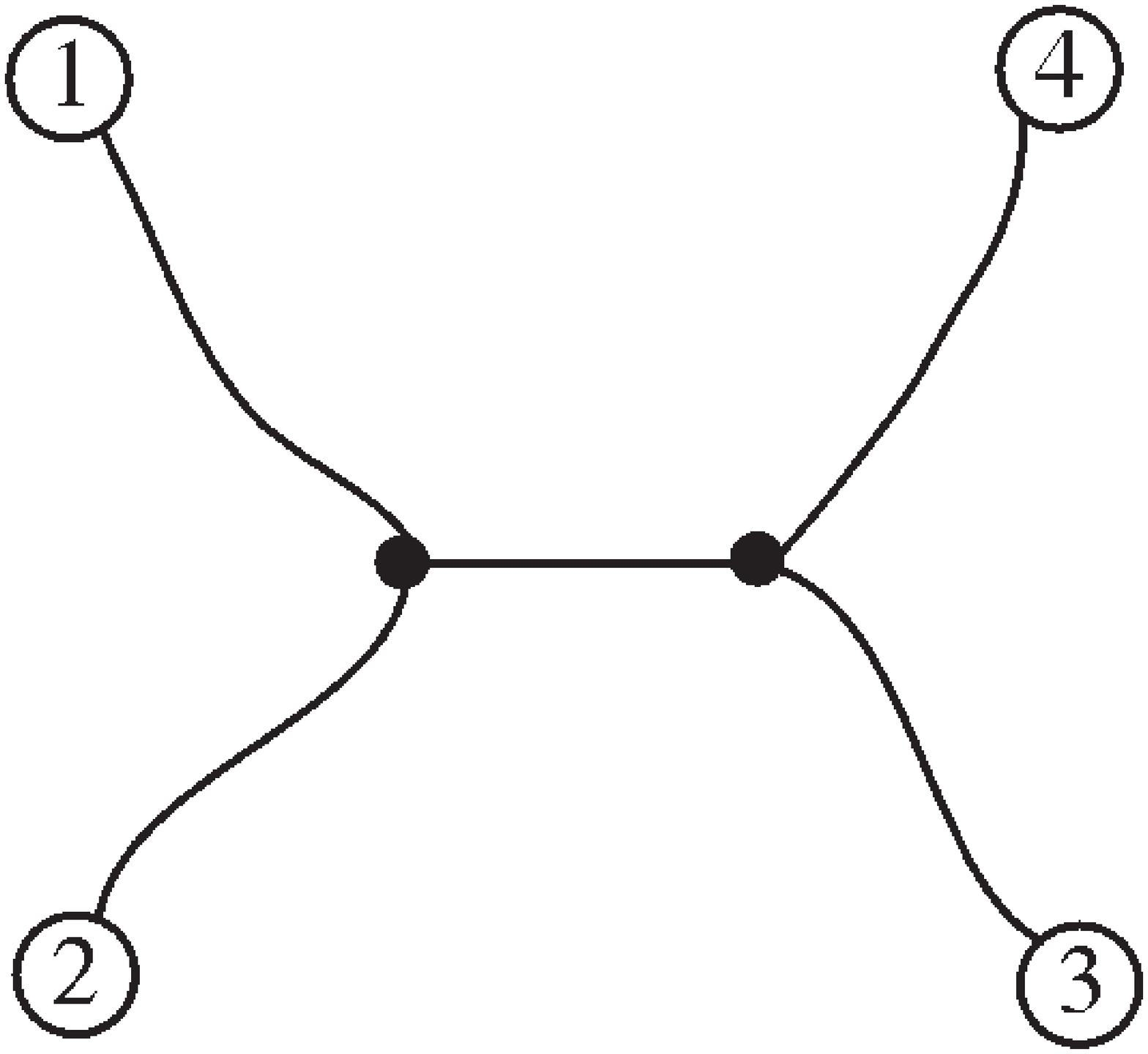}{Three binary trees with $4$ vertices of degree $1$.}{fig:manytop}{144}

So, to obtain an efficient algorithms, we have to find some {\em  a priori\/} restrictions on possible structures of minimal networks. In the next subsection we tell about the restrictions generated by geometry of boundary sets.

\subsection{Boundaries Geometry and Networks Topology}
Here we review our results from~\cite{ITUMN90} and~\cite{ITPlane}. The goal is to find some restriction on the structure of locally minimal binary trees spanning a given boundary in the plane in terms of geometry of the boundary set. To do this we need to choose or to introduce some characteristics of the network structure and of the boundary geometry.

\Pic{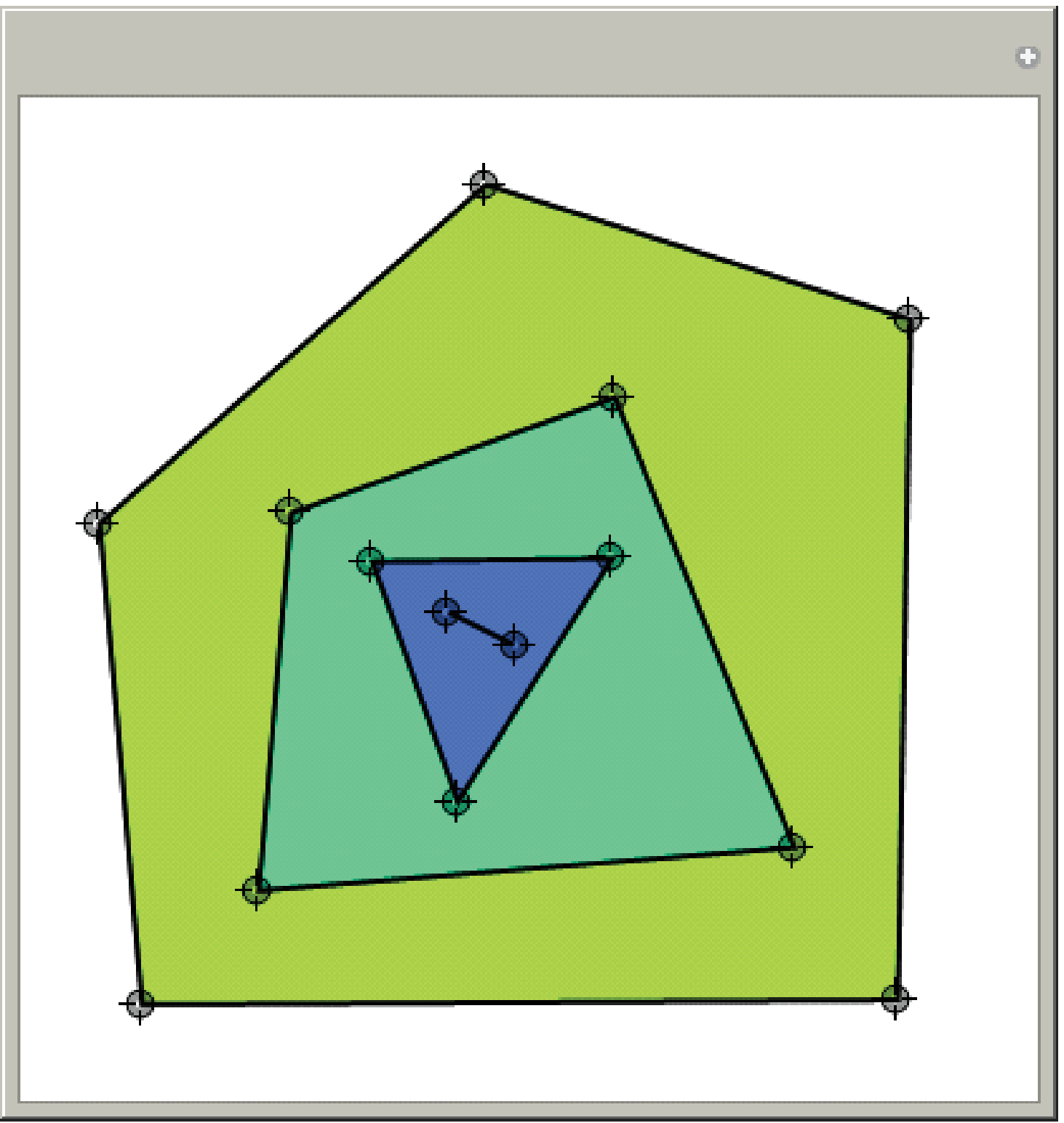}{A boundary set with $4$ convexity levels.}{fig:convex}{144}

As a characteristic of the geometry of a boundary set $M$ we take the {\em number of convexity levels $c(M)$}. Recall the definition. Let $M$ be a finite non-empty subset of the plane. Take the convex hull $\ch M$ of $M$ and assign the points from $M$ lying at the boundary of the polygon $\ch M$ to the {\em first convexity level $M^{(1)}$ of $M$}. If the set $M\setminus M^{(1)}$ is not empty, then define the {\em second convexity level $M^{(2)}$ of $M$} to be equal to the first convexity level of $M\setminus M^{(1)}$, and so on. As a result, we obtain the partition of the set $M$ into its convexity levels, and by $c(M)$ we denote the total number of this levels, see Figure~\ref{fig:convex}.

Now let us pass to definition of a characteristic describing the ``complexity'' of planar binary trees. Assume that we are given with a planar binary tree $\G$, and let the orientation of the plane be fixed. For any its two edges, say $e_1$ and $e_2$, we consider the unique path $\g$ in $\G$ starting at $e_1$ and finishing at $e_2$. All interior vertices of $\g$ are the vertices of $\G$ having degree $3$. Let us walk from $e_1$ to $e_2$ along $\g$. Then at each interior vertex of $\g$ we make either left, or right turn in $\G$. Define the value $\tw(e_1,e_2)$ to be equal to the difference between the numbers of left and right turns we have made. In other words, assign to an interior vertex of $\g$ the label $\tau=\pm1$, where $+1$ corresponds to left turns and $-1$ to right turns. Then $\tw(e_1,e_2)$ is the sum of these values, see Figure~\ref{fig:tw}. Notice that $\tw(e_1,e_2)=-\tw(e_2,e_1)$. At last, we put  $\tw\G=\max\tw(e_i,e_j)$, where the maximum is taken over all ordered pairs of edges of $\G$.

If the tree $\G$ is locally minimal, then the twisting number between any pair of its edges has a simple geometrical interpretation, see Figure~\ref{fig:tw}. Namely, since the angles between any  neighboring edges are equal to $2\pi/3$, then $\tw(e_i,e_j)$ is equal to the total angle which the oriented edge rotates by passing from $e_i$ to $e_j$, divided by $\pi/3$.

It turns out, that the twisting number of a locally minimal binary tree with a given boundary is restricted from above by a linear function on the number of convexity levels of the boundary. Namely, the following result holds.

\begin{thm}\label{th:tw_gen}
Let $\G$ be a locally minimal binary tree connecting the boundary set $M$ that coincides with the set of vertices of degree $1$ from $\G$. Then
 $$
\tw\G\le 12\bigl(c(M)-1\bigr)+5.
 $$
\end{thm}

\Pic{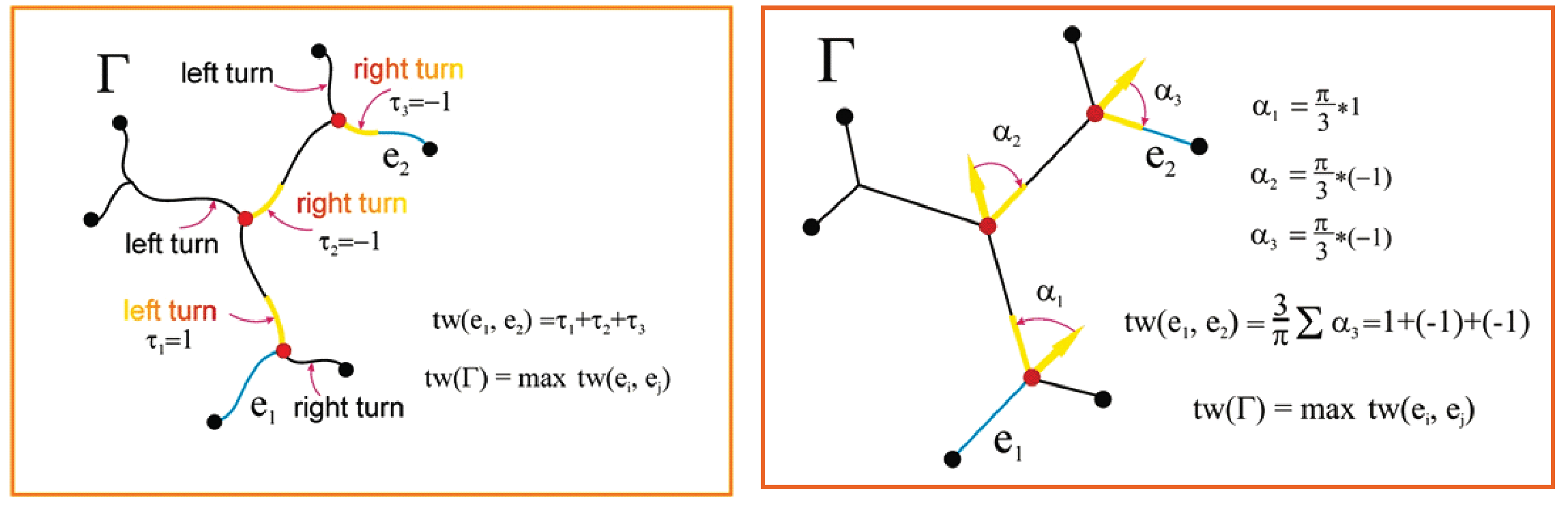}{Left and right turns in planar binary tree and its twisting number (left), and twisting number of locally minimal binary tree.}{fig:tw}{169}

The important particular case $c(M)=1$ corresponds to the vertex sets of convex polygons. Such boundaries are referred as {\em convex}.

\begin{thm}\label{th:tw5}
Let $\G$ be a locally minimal binary tree with a convex boundary. Then $\tw\G\le 5$. Conversely, any planar binary tree $\G$ with $\tw\G\le5$ is planar equivalent to a locally minimal binary tree with a convex boundary.
\end{thm}

Notice that the direct statement of Theorem~\ref{th:tw5} is rather easy to prove (it follows from the geometrical interpretation of the twisting number, easy remark that $\tw\G$ always attains at boundary edges, and the monotony of convex polygonal lines). But the converse statement is quite non-trivial. The proof obtained in~\cite{ITPlane} is based on the complete description of binary trees with twisting number at most five, obtained in terms of so-called  triangular tilings that will be discussed in the next subsection.

\begin{prb}
Estimate the number of binary trees structures with $n$ vertices of degree $1$ and twisting number at most $k$. It is more or less clear that the number is exponential on $n$ even for $k=5$, but it is interesting to obtain an exact asymptotic.
\end{prb}

\subsection{Triangular Tilings and their Applications}
It turns out that the description of planar binary trees with twisting number at most five can be effectively done in the language of planar triangulations of a special type which are referred as triangular tilings.

The correspondence between diagonal triangulations of planar convex polygons and planar binary trees is well-known: the planar dual graph of such triangulation is a binary tree, see Figure~\ref{fig:dualtr}, and each binary tree can be obtained in such a way. Here the vertices of the dual graph are centers of the triangles of the triangulation (medians intersection point) and middle points of the sides of the polygon; and edges are straight segments joining either the middle of a side with the center of the same triangle, or two centers of the triangles having a common side.

\Pic{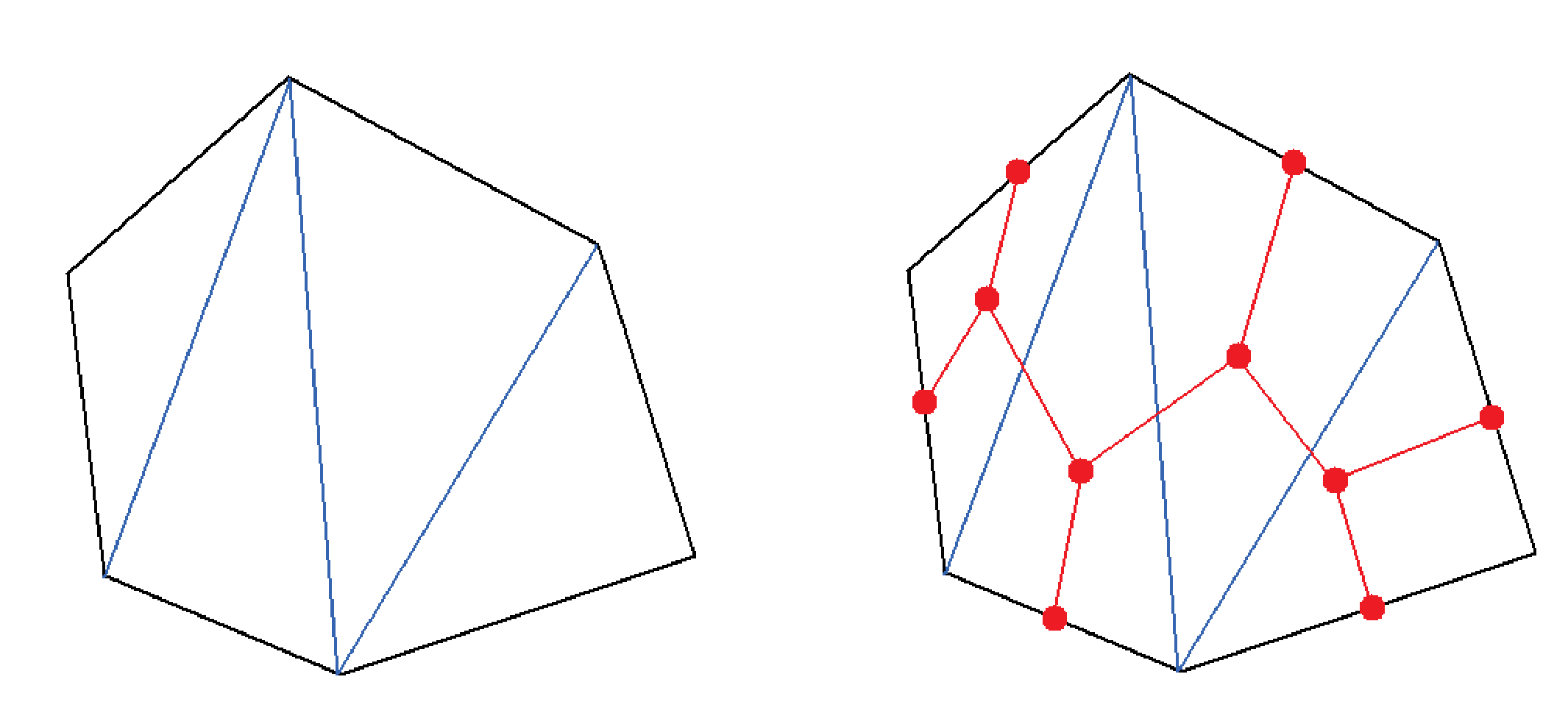}{A diagonal triangulation and corresponding planar binary tree.}{fig:dualtr}{144}

In the context of locally minimal binary trees, the most effective way to represent the diagonal triangulations is to draw them consisting of regular triangles. Such special triangulations are referred as {\em triangular tilings}. The main advantage of the tilings is that the dual binary tree constructed as described above is a locally minimal binary tree with the corresponding boundary. Therefore, tilings ``feel the geometry'' of locally minimal binary trees and turns out to be very useful in the description of such trees with small twisting numbers.

The main difficulty in constructing a triangulation consisting of regular triangles for a given binary tree is that the resulting polygon can overlap itself. An  example of such overlapping can be easily constructed from a binary tree $\G$ corresponding to the diagonal triangulation of a convex $n$-gon, $n\ge6$, all whose diagonals are incident to a common vertex. But the twisting number of such $\G$ is also at least $6$. The following result is proved in~\cite{ITPlane}.

 \begin{thm}\label{th:tiling}
The triangular tiling corresponding to any planar binary tree with twisting number less than or equal to five has no self-intersections.
 \end{thm}

Theorem~\ref{th:tiling} gives an opportunity to reduce the description of the planar binary trees with twisting number at most five to the description of the corresponding triangular tilings.

\Pic{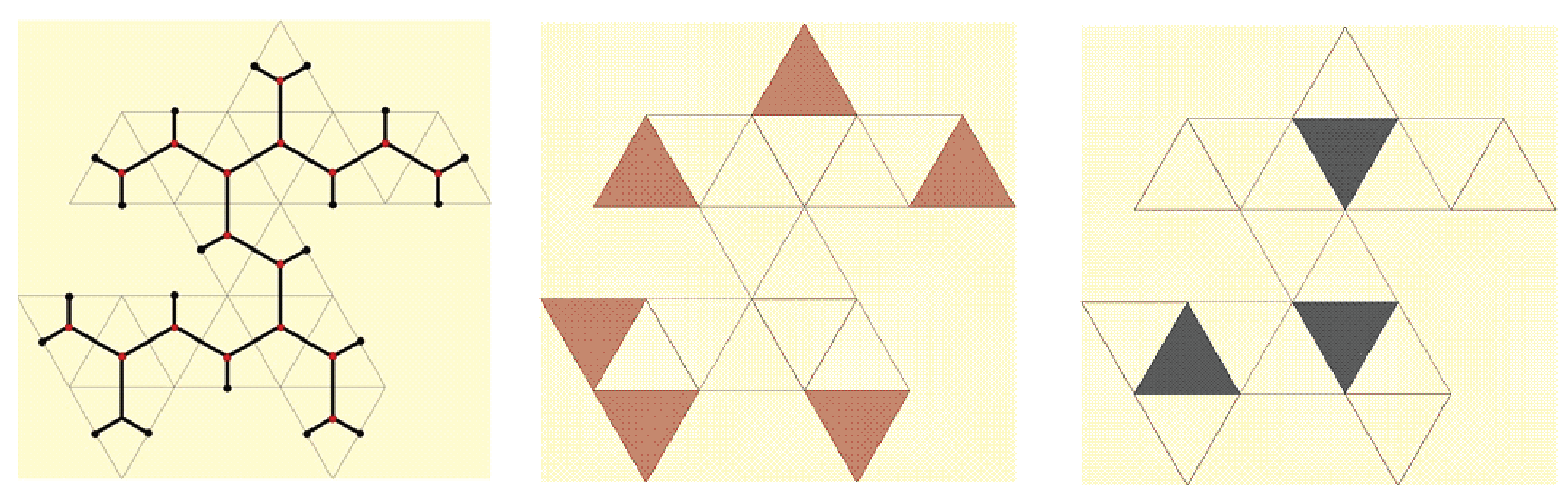}{A triangular tiling with its dual binary tree (left), its outer cells (middle) and inner cells (right).}{fig:inout}{144}

To describe all the triangular tilings whose dual binary trees have the twisting number at most five, we decompose each such tiling into elementary ``breaks''. The triangles of the tiling are referred as {\em cells}. A cell of a tiling $T$ is said to be {\em outer}, if two its sides lie at the boundary of $T$ considered as planar polygon. Further, a cell is said to be {\em inner}, if no one of its sides lies at the boundary, see Figure~\ref{fig:inout}. An outer cell adjacent to (i.e\. intersecting with by a common side) an inner cell is referred as a {\em growth of $T$}. A tiling can contains as un-paired growths, so as paired growths, see Figure~\ref{fig:scelet}.

\Pic{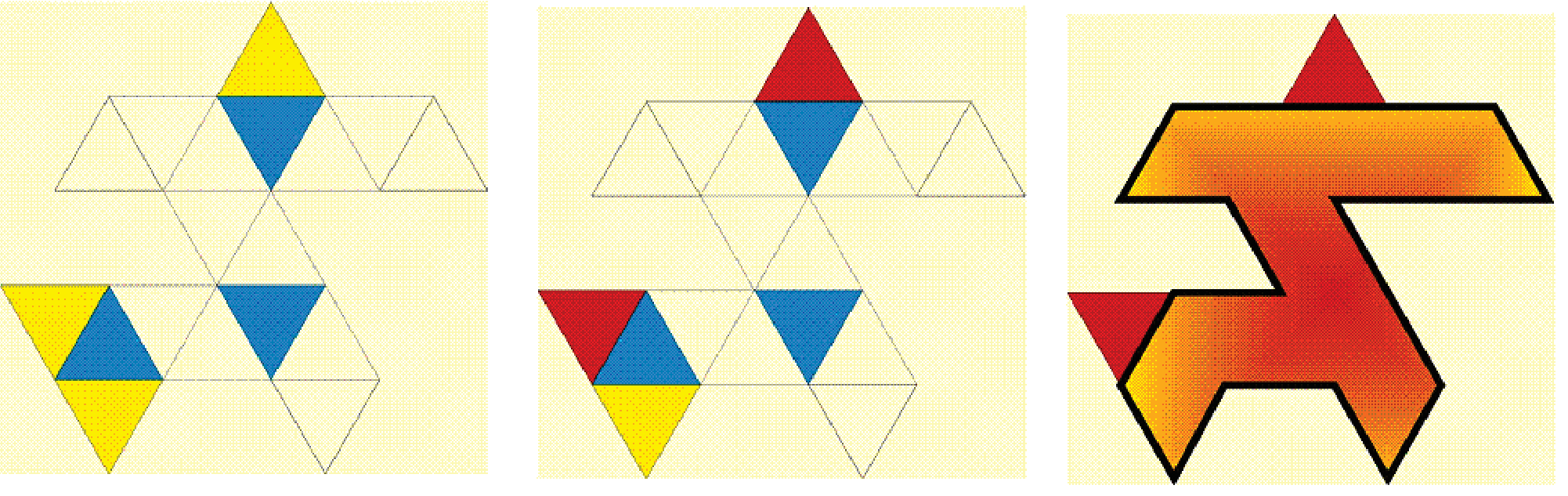}{Un-paired and paired growths of a tiling (left), growths that should be deleted to get a skeleton (middle), corresponding decomposition into skeleton and growths (right).}{fig:scelet}{144}

For each inner cell we delete exactly one growth adjacent to it, providing such growths exist. As a result, we obtain a decomposition of the initial tiling into its growths and its {\em skeleton\/} (a tiling without growths). Notice, that such a decomposition is not unique.

It turns out that the skeletons of the tilings whose dual binary trees have twisting number at most five can be described easily. Also, the possible location of growthes in such tilings on their skeletons also can be described. The details can be found in~\cite{ITPlane} or~\cite{ITBookWP}. Here we only formulate the skeletons describing Theorem and include several examples of its application.

Inner cells of a skeleton $S$ are organized into so-called {\em branching points}, see Figure~\ref{fig:code}. After the branching points deleting, the skeleton is partitioned into {\em linear parts}. Each linear part contains at most one outer cell. Construct a graph $C(S)$  referred as the {\em code of the skeleton $S$} as follows: the vertex set of $C(S)$ is the set of its branching points and of the outer cells of its linear parts. The edges correspond to the linear parts, see Figure~\ref{fig:code}.

\Pic{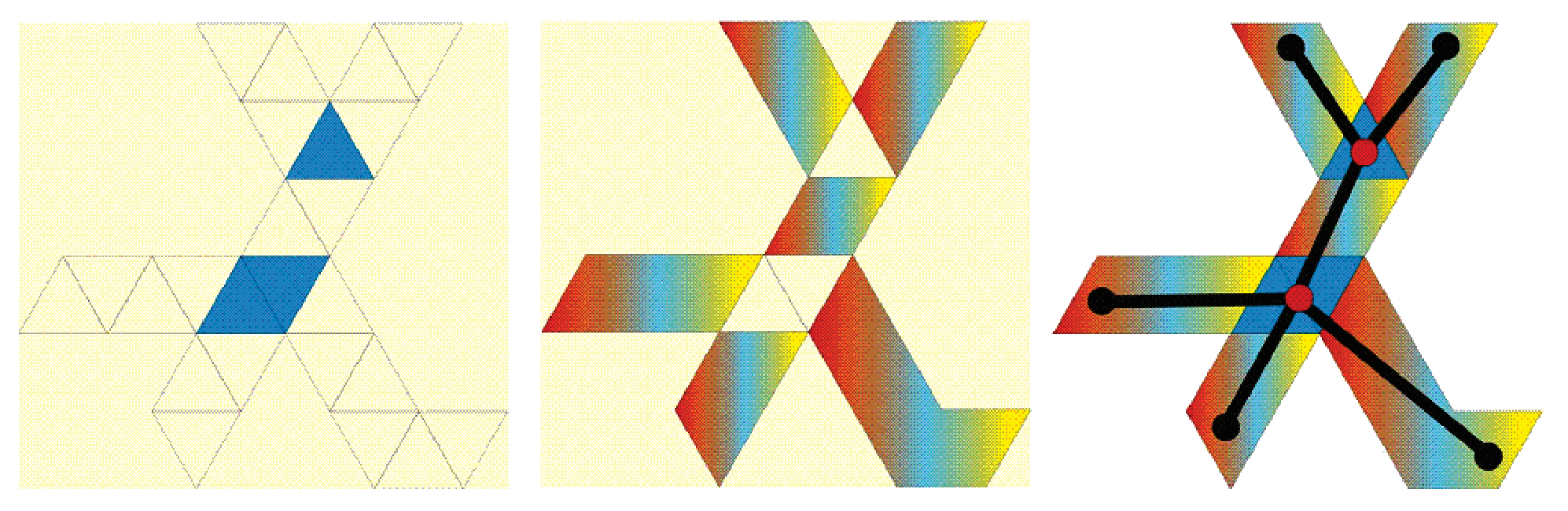}{Branching points (left), linear parts (middle), and code of a skeleton (right).}{fig:code}{144}

The following result is proved in~\cite{ITPlane}.

 \begin{thm}
Consider all skeletons whose dual graphs twisting numbers are at most $5$ and for each of  these skeletons construct its code. Then, up to planar equivalence, we obtain all planar graphs with at most $6$ vertices of degree $1$ and without vertices of degree $2$. In particular, every such skeleton contains at most $4$ branching points and at most $9$ linear parts.
 \end{thm}

All possible codes of such skeletons are depicted in Figure~\ref{fig:allcod}.

\Pic{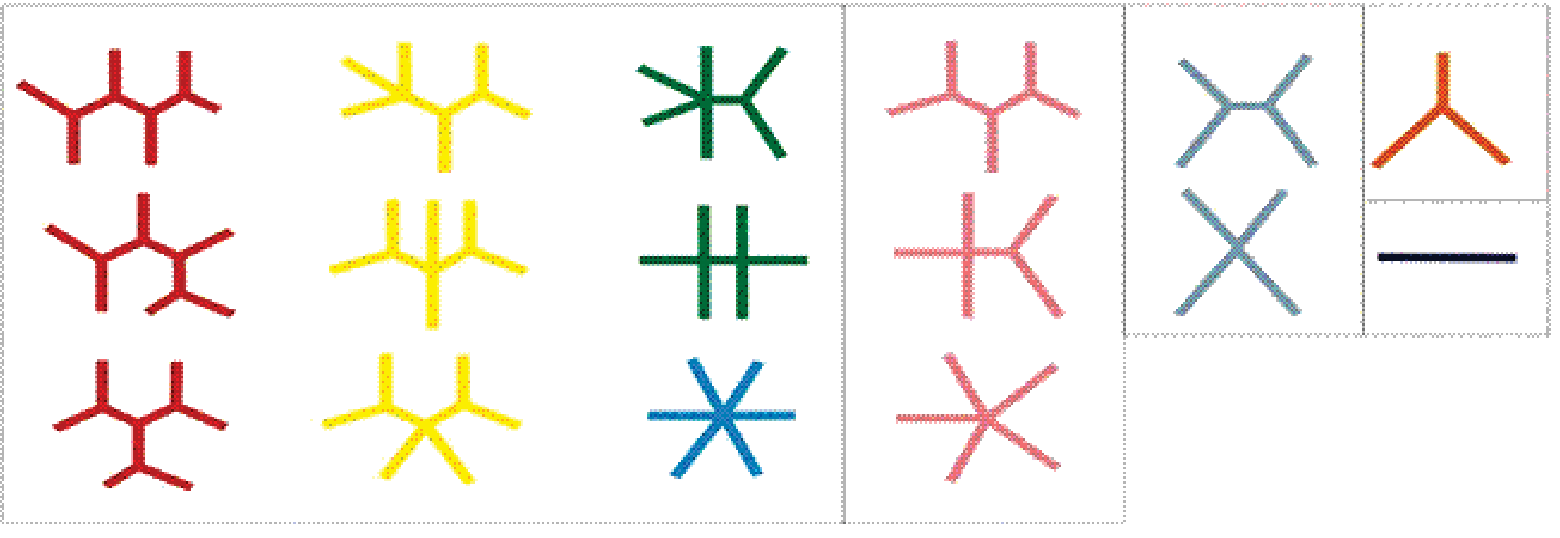}{All possible codes of skeletons whose dual binary trees have twisting number at most $5$.}{fig:allcod}{144}

This description of skeletons and corresponding tilings obtained in~\cite{ITPlane}, was applied to the proof of inverse (non-trivial) statement of Theorem~\ref{th:tw5}. In some sense, the proof obtained in~\cite{ITPlane} is constructive: for each tiling under consideration a corresponding locally minimal binary tree with a convex boundary is constructed.

Another application is a description of all possible binary trees of the skeleton type that can be realized as locally minimal binary trees connecting the vertex set of a regular polygon. It turns out, see details in~\cite{ITBookWP}, that there are $2$ infinite families of such trees and $1$ finite family. The representatives of these networks together with the  corresponding skeletons are shown in Figure~\ref{fig:ngon}.

\Pic{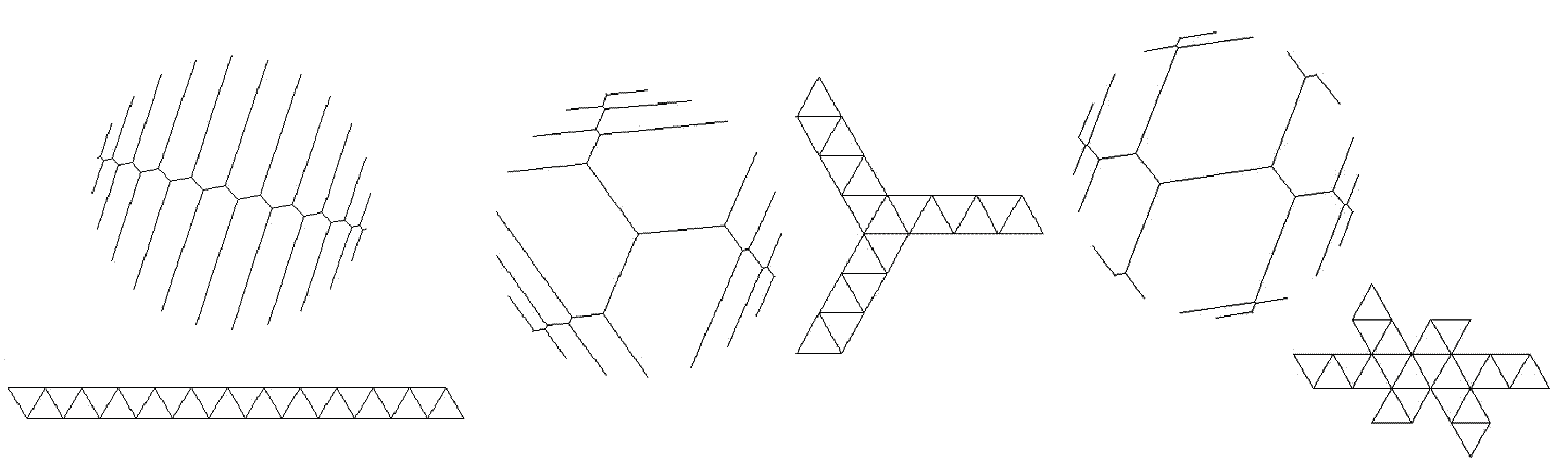}{Two infinite families of locally minimal binary trees connecting vertex sets of regular polygons, existing for any regular $n$-gon (left) and for $3k+6$-gons (middle), and the unique finite family, existing for $24$-, $30$-, $36$- and $42$-gons only (right).}{fig:ngon}{144}

\section{Steiner Ratio}\label{sec:sr}
\markright{\thesection.~Steiner Ratio.}
As we have already discussed in the previous Section, the problem of finding a shortest tree connecting a given boundary set is exponential even in two-dimensional Euclidean plane. On the other hand, in practice it is necessary to solve transportation problems of this kind for several thousands boundary points many times a day. Therefore, in practice some heuristical algorithms are used. One of the most popular heuristics for a shortest tree is corresponding minimal spanning tree. But using such approximate solutions instead of exact one it is important to know the value of possible error appearing under the approximation. The {\em Steiner ratio of a metric space\/} is just the measure of maximal possible relative error for the approximation of a shortest tree  by the corresponding minimal spanning tree.

\subsection{Steiner Ratio of a Metric Space}
Let $M$ be a finite subset of a metric space $(X, \rho)$, and assume that $|M|\ge2$. We put $\sr M=\smt(M)/\mst(M)$. Evidently, $\sr M\le 1$. The next statement is also easy to prove.

 \begin{ass}
For any metric space $(X,\r)$ and any its finite subset $M\subset X$, $|M|\ge2$, the inequality $\sr M>1/2$ is valid.
 \end{ass}

\begin{proof}
Let $G$ be a Steiner tree connecting $M$. Consider an arbitrary embedding of the graph $G$ into the plane, walk around $G$ in the plane and list consecutive paths forming this tour and joining consecutive boundary vertices from $M$. The length of each such path $\g_{PQ}$ joining boundary vertices $PQ$, i.e\. the sum of the lengthes of its edges, is more than or equal to the distance $\r(P,Q)$, due to the triangle inequality. Consider the cyclic path in the complete graph with vertex set $M$ consisting of edges formed by the pairs of consecutive vertices from the tour, and let $T$ be a spanning tree on $M$ contained in this path. It is clear, that $\r(T)<\sum_{(P,Q)}\r(\g_{PQ})$, where the summation is taken over all the pairs of consecutive vertices of the tour. On the other hand, each edge of the tree $G$ belongs to exactly two such paths, hence $\sum_{(P,Q)}\r(\g_{PQ})=2\r(G)$. So, we have $\sr(M)\ge\r(G)/\r(T)>1/2$. The Assertion is proved.
\end{proof}

The value $\sr(M)$ is the relative error appearing under approximation of the length of a shortest tree for a given set $M$ by the length of a minimal spanning tree. The {\em Steiner ratio of a metric space $(X,\rho)$} is defined as the value $\sr(X)=\inf_{M\subset X}\sr(M)$, where the infimum is taken over all finite subsets $M$, $|M|\ge2$ of the metric space $X$. So, the Steiner ratio of $X$ is the value of the relative error in the worse possible case.

\begin{cor}
For arbitrary metric space $(X,\rho)$ the inequality $1/2\le\sr(X)\le1$ is valid.
\end{cor}

\begin{exe}
Verify, that for any $r\in[1/2,1]$ there exists a metric space $(X,\r)$ with $\sr(X)=r$, see corresponding examples in~\cite{ITBookRFFI}.
\end{exe}

Sometimes, it is convenient to consider so-called {\em Steiner ratios $\sr_n(X)$ of degree $n$}, where $n\ge2$ is an integer, which are defined as follows: $\sr_n(X)=\inf_{M\subset X,|M|\le n}\sr(M)$. Evidently, $\sr_2(X)=1$. It is also clear that $\sr(X)=\inf_n\sr_n(X)$.

Steiner ratio was firstly defined for the Euclidean plane in~\cite{GilPol}, and during the following years the problem of Steiner ratio calculation is one of the most attractive, interesting and difficult problems in geometrical optimization. A short review can be found in~\cite{ITBookRFFI} and in~\cite{CiesBookSR}. One of the most famous stories here is connected with several attempts to prove so-called {\em Gilbert--Pollack Conjecture}, see~\cite{GilPol}, saying that $\sr(\R^2,\r_2)=\sqrt{3}/2$, where $\r_2$ stands for the Euclidean metric, and hence $\sr(\R^2,\r_2)$ is attained at the vertex set of a regular triangle, see Figure~\ref{fig:tri}. In 1990s D.\,Z.~Du and F.\,K.~Hwang announced that they proved the Steiner Ratio Gilbert--Pollak Conjecture~\cite{DuHwang90}, and their proof was published in Algorithmica~\cite{DuHwang}. In spite of the appealing ideas of the paper, the questions concerning the proof appeared just after the publication,
because the text did not appear formal. And about 2003--2005 it becomes clear that the gaps in the D.\,Z.~Du and F.\,K.~Hwang work are too deep and can not be repaired, see detail in~\cite{ITAlg}.

\subsection{Steiner Ratio of Small Degrees for Euclidean Plane}
Gilbert and Pollack calculated $\sr_3(\R^2,\r_2)$ in their paper~\cite{GilPol}. We include their proof here.

Since the Steiner ratio of a regular triangle is equal to $\sqrt{3}/2$, then $\sr_3(\R^2,\r_2)\le\sqrt{3}/2$, so we just need to prove the opposite inequality. To do this, consider a triangle $ABC$ in the plane. If one of its angles is more than or equal to $120^\c$, then the shortest tree coincides with minimal spanning tree, so in this case $\sr(ABC)=1$. So it suffices to consider the case when all the angles of the triangle are less than $120^\c$.

Let $S$ be the Torricelli point of the triangle $ABC$. Show firstly that $|AS|\le|BS|$, if and only if $|BC|\ge|AC|$, i.e\. the shortest edge of the Steiner minimal tree lies opposite with the longest side of the triangle. The proof is shown in Figure~\ref{fig:sr3}, left. Indeed, if $|BS|<|AS|$, then we take the point $B'\in[S,B]$ with $|SB'|=|SA|$, hence $|CB'|=|CA|$ due to symmetry and $|CB'|<|CB|$ because $B'\ge120^c$. Conversely, if $|BC|>|B'C|$, then there exists $B'\in[B,S]$ with $|CB'|=|CA|$, because $|BC|>|CA|>|SC|$. Then $|AS|=|SB'|<|SB|$.

\Pic{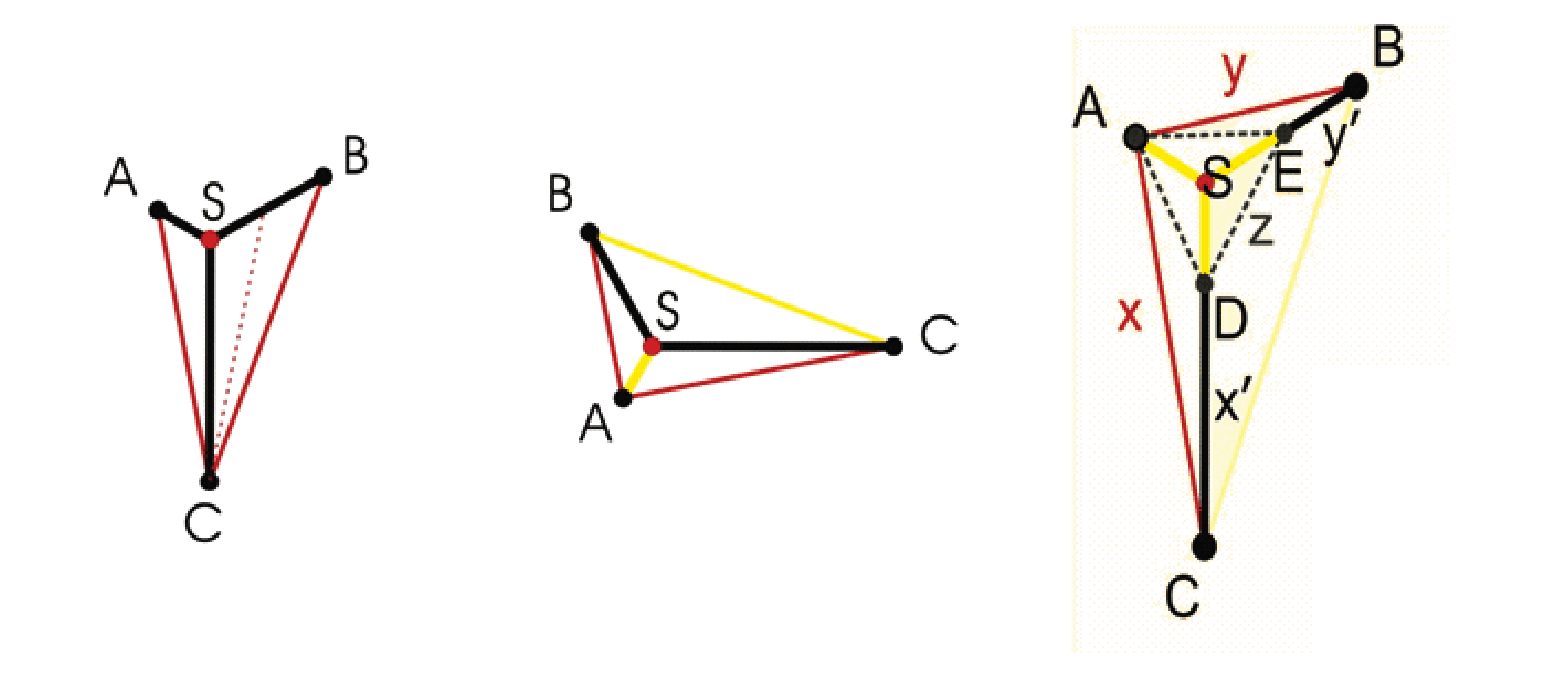}{To the calculation of $\sr_3(\R^2,\r_2)$.}{fig:sr3}{144}

Thus, the two-edges tree $T=[A,B]\cup[B,C]$ is a minimal spanning tree for $ABC$, if and only if $BC$ is the longest side of $ABC$, if and only if $|AS|\le|BS|$ and $|AS|\le|CS|$. Consider the points $E\in[B,S]$ and $D\in[C,S]$, such that $|AS|=|ES|=|DS|$, and put $x=|AC|$, $y=|AB|$, $z=|DE|=|AD|=|AE|$, and $x'=|CD|$, $y'=|EB|$. Then $|SA|=|SE|=|SD|=z/\sqrt{3}$ and
$$
\smt(M)=3|SA|+|DC|+|EB|=\sqrt{3}z+x'+y' \qquad\text{and}\qquad \mst(M)=x+y,
$$
where $M$ stands for the set $\{A,B,C\}$.
But $x\le x'+z$ and $y\le y'+z$, due to the triangle inequality, and hence
$$
\sr(M)=\frac{\sqrt{3}z+x'+y'}{x+y}\ge\frac{\sqrt{3}z+x'+y'}{x'+z+y'+z}=\frac{\sqrt{3}z+x'+y'}{x'+y'+2z}\ge\frac{\sqrt{3}}{2}.
$$
Thus, we proved the following statement.

\begin{ass}
The following relation is valid: $\sr_3(\R^2,\r_2)=\sqrt{3}/2$.
\end{ass}

\begin{rk}
For small $n$ it is already proved that $\sr_n(\R^2,\r_2)=\sqrt{3}/2$ (recently O.~de~Wet proved it for $n\le7$, see~\cite{deWet}). The proof of de Wet is based on the analysis of Du and Hwand method from~\cite{DuHwang} and understanding that it works for boundary sets with $n\le 7$ points.  Also in 60th several lower bounds for $\sr(\R^2,\r_2)$ were obtained, and the best of them is worse than $\sqrt{3}/2$ in the third digit only.
\end{rk}

\begin{prb}
Very attractive problem is to prove that $\sr(\R^2,\r_2)=\sqrt{3}/2$, i.e\. to prove Gilbert--Pollack Conjecture. The attempts to repair the proof of Du and Hwang have remained unsuccessful, so some fresh ideas are necessary here.
\end{prb}

\subsection{Steiner Ratio of Other Euclidean Spaces and Riemannian Manifolds}

The following result is evident, but useful.

\begin{ass}
If $Y$ is a subspace of a metric space $X$, i.e\. the distance function on $Y$ is the restriction of the distance function of $X$, then $\sr(Y)\ge\sr(X)$.
\end{ass}

This implies, that $\sr(\R^n,\r_2)\le\sr(\R^2,\r_2)\le\sqrt{3}/2$. Recall that Gilbert--Pollack conjecture implies that the Steiner ratio of Euclidean plane attains at the vertex set of a regular triangle. In multidimensional case the situation is more complicated. The following result was obtained by Du and Smith~\cite{DuSmith}

\begin{ass}
If $M\subset\R^n$ is the vertex set of a regular $n$-dimensional simplex, then $\sr(M)>\sr(\R^n,\r_2)$ for $n\ge3$.
\end{ass}

\begin{proof}
Consider the boundary set $P$ in $\R^{n+1}$, consisting of the following $1+n(n+1)$ points: one point $(0,\ldots,0)$ and $n(n+1)$ points all whose coordinates except two are zero, one is equal to $1$, and the remaining one is $-1$. It is clear that $P$ is a subset of  $n$-dimensional plane defined by the next linear condition: sum of all coordinates is equal to zero. Represent $P$ as the union of the subsets $P^i=\{x\in P\mid x^i=1\}\cup \{(0,\ldots,0)\}$. Notice that each set $P^i$, $i=1,\ldots,n+1$, consists of $n+1$ points and forms the vertex set of an regular $n$-dimensional simplex (to see that it suffices to verify that all the distances between the pairs of points from $P^i$ are the same and are equal to $\sqrt{2}$). The configuration of $7$ points in $\R^3$ is shown in Figure~\ref{fig:DuSmith} (this case is not important for us, but it is easy to draw). Now, $\mst(P)=(n+1)\mst(P^i)$, but for $n\ge 3$ we conclude that $\smt(P)<(n+1)\smt(P^i)$, because the degree of the vertex $(0,\ldots,0)$ in the corresponding network which is the union of the shortest networks for $P^i$ is equal to $n+1\ge4$ that is impossible in the shortest network due to the Local Structure Theorem~\ref{th:LStr}. So,
$$
\sr(P)=\smt(P)/\mst(P)<\frac{(n+1)\smt(P^i)}{(n+1)\mst(P^i)}=\sr(P^i).
$$
\end{proof}

\Pic{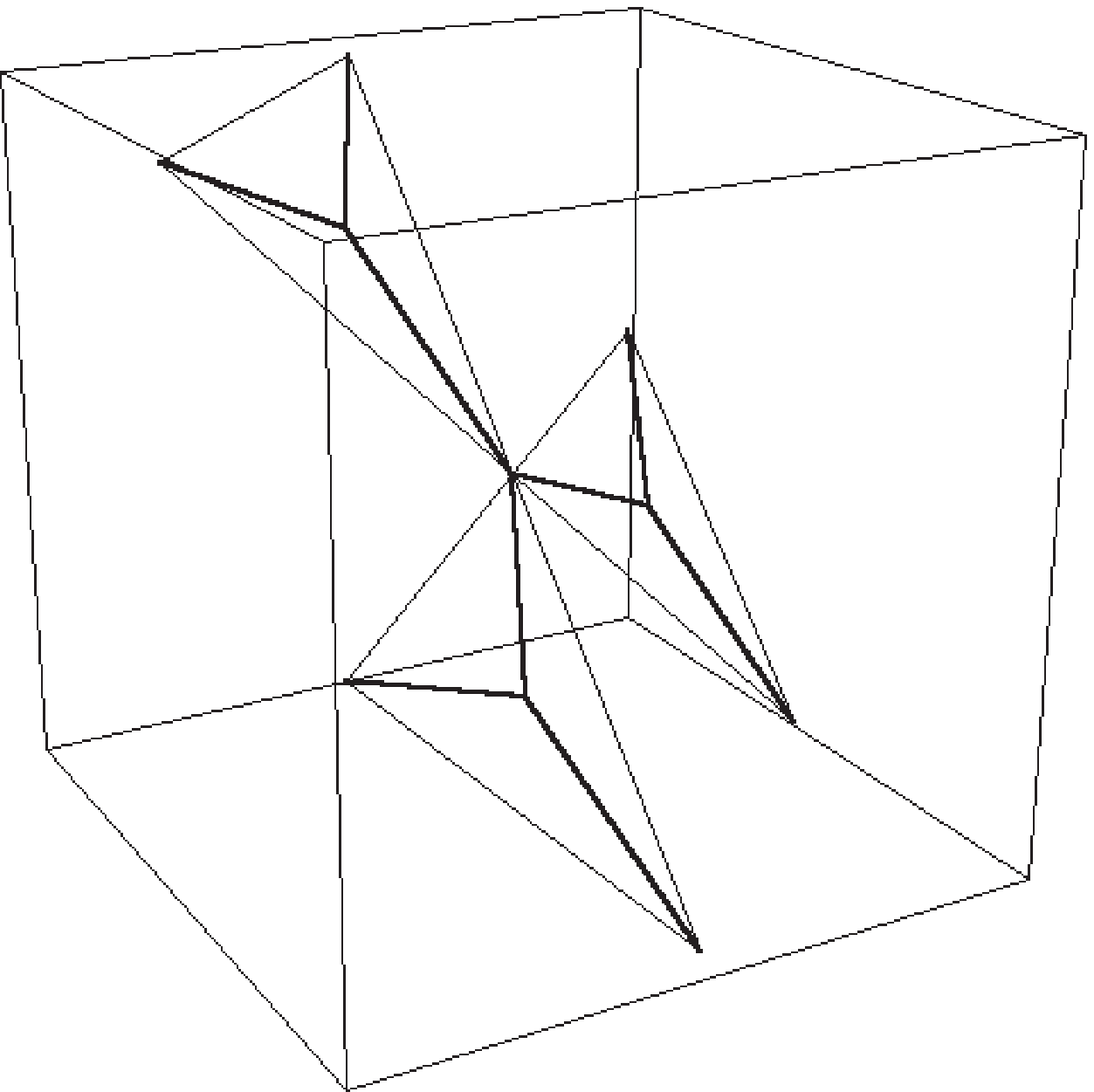}{Construction of the set $P$ in $\R^3$ (non-interesting but visual case $n=2$).}{fig:DuSmith}{144}

Taking as a heuristic for the length of a shortest tree connecting the vertex set of regular simplex the length of the network joining the center of the simplex with all its vertices we get the following estimate.

\begin{cor}
For any $n\ge 3$ the upper estimate
$$
\sr(\R^n,\r_2)<\sqrt{\frac{1}2+\frac{1}{2n}}
$$
is valid.
\end{cor}

One of the best general low estimates is obtained by Graham and Hwang in~\cite{GH}.

\begin{ass}
For any $n\ge2$ the lower estimate $1/\sqrt3\le\sr(\R^n,\r_2)$ is valid.
\end{ass}

The best known upper estimate for $\R^3$ is obtained by Smith and Smith~\cite{SmSm}. It is attained at  an infinite boundary set which is known as ``Smith sausage'' and depicted in Figure~\ref{fig:SmSm}. The corresponding value, obtained as the limit of the ratios for finite fragments, is as follows:
$$
\sqrt{\frac{283}{700}-\frac{3\sqrt{21}}{700}+\frac{9\sqrt{22-2\sqrt{21}}}{140}}.
$$
Notice that the idea of an infinite set is based on a deep result of Du and Smith estimating from below the number of points in a subset $M$ of $\R^n$ such that $\sr(M)=\sr(\R^n,\r_2)$ by a function $f(n)$ rapidly increasing on $n$, see details in~\cite{DuSmith}. For example, $f(50)=53$, but $f(200)=3\,481\,911$. Therefore, it is difficult to expect to guess a finite set $M$ in $\R^n$ with $\sr(M)=\sr(\R^n,\r_2)$ for large $n$.

\Pic{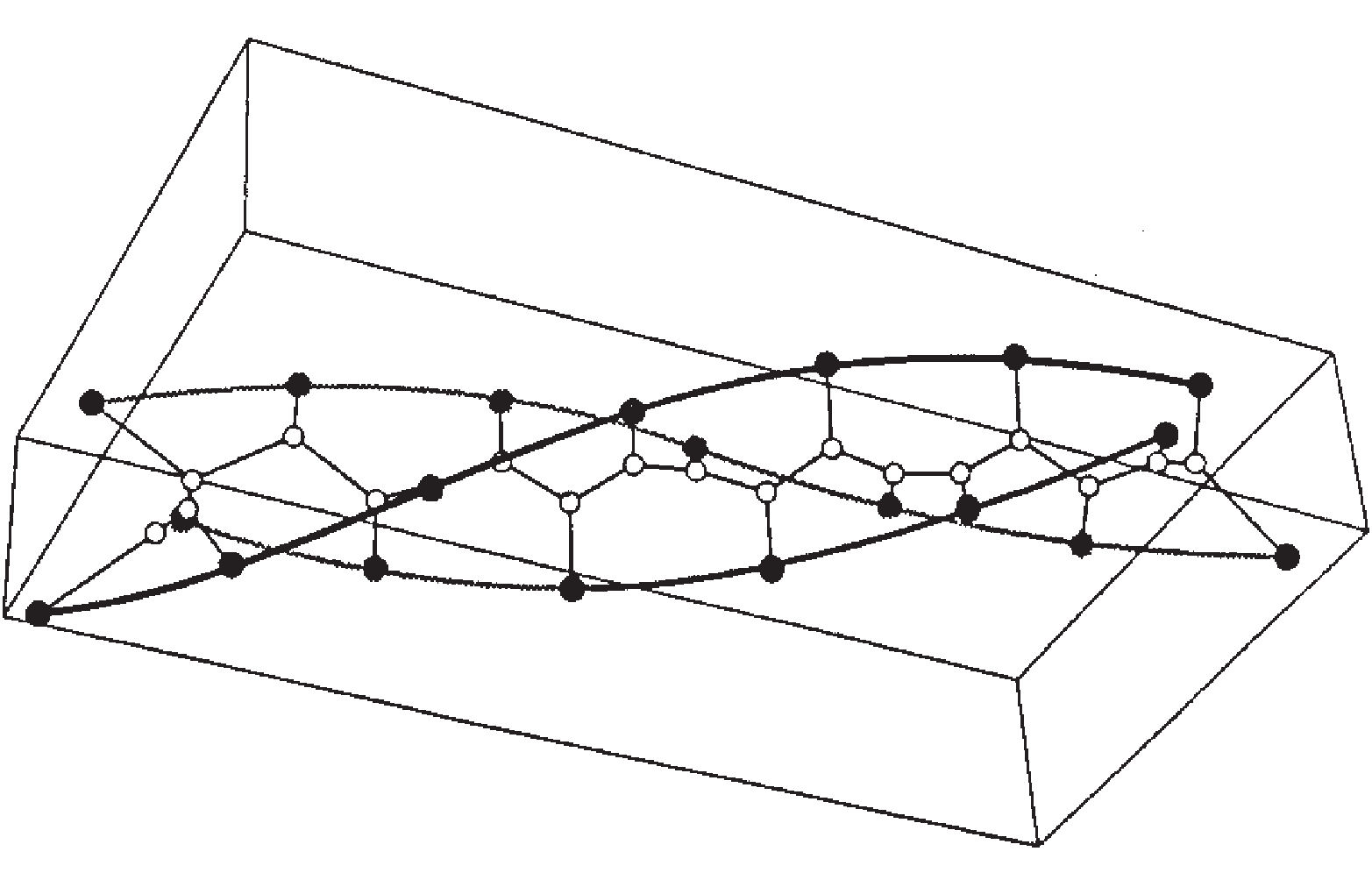}{A finite fragment of infinite ``Smith sausage''.}{fig:SmSm}{144}

Recently, the Steiner ratio of the Lobachevskii plane, and hence, of any Lobachevskii space has bin calculated by Innami and Kim, see~\cite{InKim}.

\begin{thm}
Steiner ratio of Lobachevskii space $L^n$ for any $n\ge2$ is equal to $1/2$.
\end{thm}

For general Riemannian manifold Ivanov, Cieslik and Tuzhilin, see~\cite{ITC}, obtained the following general result.

\begin{thm}
The Steiner ratio of $n$-dimensional Riemannian manifold is less than or equal to the Steiner ratio of the Euclidean space $\R^n$.
\end{thm}

\section{Minimal  Fillings}\label{sec:mf}
\markright{\thesection.~Minimal Fillings.}
This Section is devoted to minimal fillins, the third kind of optimal connections discussed in the Introduction. This problem appeared as a result of a synthesis of two classical problems: the Steiner problem on the shortest networks (it is discussed in Sections~\ref{sec:smt} and~\ref{sec:sr}), and Gromov's problem on minimal fillings. 

The concept of a minimal filling appeared in papers of Gromov, see~\cite{Gromov}. Let $M$ be a manifold endowed with a distance function  $\rho$. Consider all possible films $W$ spanning $M$, i.e., compact manifolds with the boundary $M$. Consider on  $W$ a distance function $d$ that does not decrease the distances between points in $M$. Such a metric space ${\mathcal W}=(W,d)$ is called a \emph{filling\/} of the metric space ${\mathcal M}=(M,\rho)$, see example in Figure~\ref{fig:MinFil}. The Gromov Problem consists in calculating the infimum of the volumes of the fillings and describing the spaces ${\mathcal W}$ which this infimum is achieved at (such spaces are called \emph{minimal fillings\/}).

\Pic{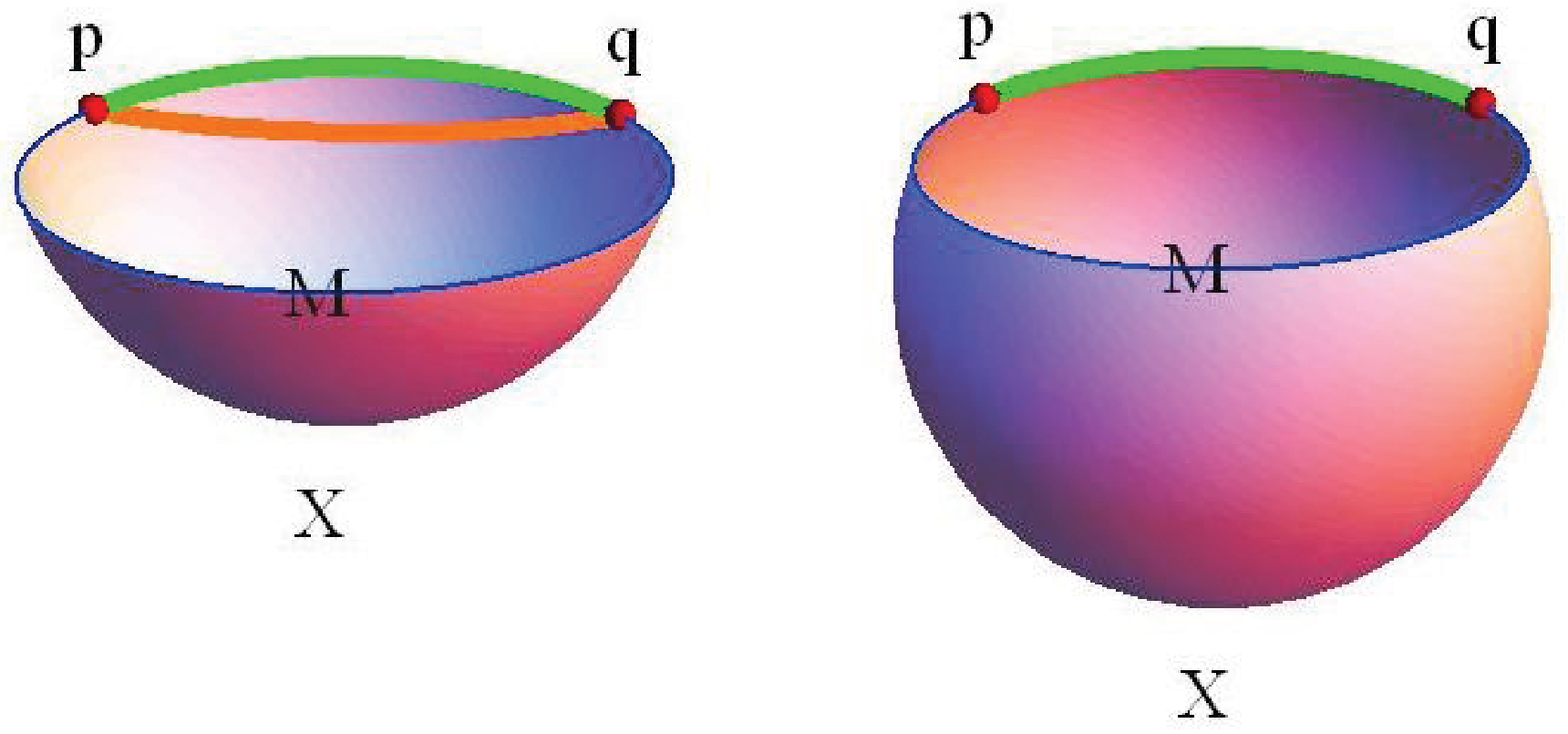}{The space $M$ is the circle $S^1$ with arc-metric. The films $X$ in the both Figures are parts of the standard sphere containing $M$ as a parallel. The left film $X$ is not a filling since the distance between the points $p$ and $q$ in $X$ is less than in $M$ (the shortest path is depicted). The right film $X$ is a filling of $M$.}{fig:MinFil}{144}

In the scope of Steiner problem, it is natural to consider $M$ as a finite metric space. Then the possible fillings are metric spaces having the structure of one-dimensional stratified manifolds which can be considered as graphs whose edges have nonnegative weights. This leads to the following particular case of generalized Gromov problem.

Let $M$ be an arbitrary finite set, and $G=(V,E)$ be a connected graph. We say, that $G$ \emph{connects $M$} or {\em joins $M$}, if $M\subset V$. Now, let ${\mathcal M}=(M,\rho)$ be a finite metric space, $G=(V,E)$ be a connected graph joining $M$, and $\omega\colon E\to{\mathbb R}_+$ is a mapping into non-negative numbers, which is usually referred as a \emph{weight function\/} and which generates the \emph{weighted graph} ${\mathcal G}=(G,\omega)$. The function $\omega$ generates on $V$ the pseudo-metric $d_\omega$ (some distances in a pseudo-metric can be equal to zero), namely, the $d_\omega$-distance between the vertices of the graph  ${\mathcal G}$ is defined as the least possible weight of the paths in ${\mathcal G}$ joining these vertices. If for any two points $p$ and $q$ from $M$ the inequality $\rho(p,q)\le d_\omega(p,q)$ holds, then the weighted graph ${\mathcal G}$ is called a \emph{ filling\/} of the space ${\mathcal M}$, and the graph $G$ is referred as the \emph{type\/} of this filing. The value $\operatorname{mf}({\mathcal M})=\inf\omega({\mathcal G})$, where the infimum is taken over all the fillings ${\mathcal G}$ of the space ${\mathcal M}$ is the \emph{weight of minimal filling}, and each filling ${\mathcal G}$ such that  $\omega({\mathcal G})=\operatorname{mf}({\mathcal M})$ is called a \emph{minimal filling}.

\subsection{Parametric Networks and Optimal Connection Problems}
Here we give a common view on Steiner problem and minimal filling problem in terms of so-called parametric networks in a general metric space.

Let ${\mathcal X}=(X,d)$ be a metric space and $G=(V,E)$ be an arbitrary connected graph. Any mapping $\Gamma\colon V\to X$ is called a \emph{ network in ${\mathcal X}$ parameterized by the graph $G=(V,E)$}, or a \emph{network of the type $G$}. The \emph{vertices} and \emph{edges\/} of the network $\Gamma$ are the restrictions of the mapping $\Gamma$ onto the vertices and edges of the graph $G$, respectively. The \emph{length of the edge\/} $\Gamma\colon vw\to X$ is the value $d\bigl(\Gamma(v),\Gamma(w)\bigr)$, and the \emph{length $d(\Gamma)$ of the network $\Gamma$} is the sum of lengths of all its edges. We shall consider various boundary value problems for graphs. To do that, we fix some subsets $\partial G$ of the vertex sets $V$ of our graphs $G=(V,E)$, and we call such $\partial G$ the \emph{boundaries}. We always suppose that in each graph under consideration a boundary, possibly, an empty one, is chosen. The \emph{boundary $\partial\Gamma$ of a network $\Gamma$} is the restriction of $\Gamma$ onto $\partial G$. If $M\subset X$ is finite and $M\subset\Gamma(V)$, then we say that the network $\Gamma$ \emph{joins\/} or \emph{connects the set $M$}. The vertices of graphs and networks which are not boundary ones are called \emph{interior\/} vertices. The value
 $$
\operatorname{smt}(M)=\inf\bigl\{d(\Gamma)\mid\text{$\Gamma$ is a network joining $M$}\bigr\}
 $$
is called the \emph{length of shortest network for $M$}. Notice that the network $\Gamma$ which joins $M$
and satisfies $d(\Gamma)=\operatorname{smt}(M)$ may not exist, see~\cite{ITLup} and~\cite{Borod} for nontrivial examples.
If such a network exists, it is called a \emph{shortest network connecting $M$}, or \emph{for $M$}.
One variant of the Steiner problem is to describe the shortest networks for finite subsets of metric spaces.%
 \footnote{%
The denotation $\operatorname{smt}$ is an acronym for ``Steiner Minimal Tree'' which is a synonym for the shortest network whose edges are non-degenerate and, thus, it must be a tree.
 }

Now let us define minimal parametric networks in a metric space ${\mathcal X}=(X,d)$. Let $G=(V,E)$ be a connected graph with some boundary $\partial G$, and let $\varphi\colon \partial G\to X$ be a mapping. By $[G,\varphi]$ we denote the set of all networks $\Gamma\colon V\to X$ of the type $G$ such that $\partial\Gamma=\varphi$. We put
 $$
\operatorname{mpn}(G,\varphi)=\inf_{\Gamma\in[G,\varphi]}d(\Gamma)
 $$
and we call this value the \emph{length of minimal parametric network}. If there exists a network $\Gamma\in[G,\varphi]$ such that $d(\Gamma)=\operatorname{mpn}(G,\varphi)$, then $\Gamma$ is called a \emph{minimal parametric network of the type $G$ with the boundary $\varphi$}.

\begin {ass}
Let ${\mathcal X}=(X,d)$ be an arbitrary metric space and $M$ be a finite subset of $X$. Then
 $$
\operatorname{smt}(M)=\inf\bigl\{\operatorname{mpn}(G,\varphi)\mid\varphi(\partial G)=M\bigr\},
 $$
where the infimum is taken over all connected graphs $G$ with a boundary $\partial G$ and all mappings $\varphi\colon\partial G\to X$ with $\varphi(\partial G)=M$.
\end {ass}

Thus, as in the case of the plane, the problem of calculating the length of the shortest network is reduced to investigation of minimal parametric networks.

Let ${\mathcal M}=(M,\rho)$ be a finite metric space and $G=(V,E)$ be an arbitrary connected graph connecting $M$. In this case we always assume that the boundary of such $G$ is fixed and equal to $M$. By $\Omega({\mathcal M},G)$ we denote the set of all weight functions $\omega\colon E\to{\mathbb R}$ such that $(G,\omega)$ is a filling of the space ${\mathcal M}$. We put
 $$
\operatorname{mpf}({\mathcal M},G)=\inf_{\omega\in\Omega({\mathcal M},G)}\omega(G)
 $$
and we call this value the \emph{weight of minimal parametric filling of the type $G$ for the space ${\mathcal M}$}. If there exists a weight function $\omega\in\Omega({\mathcal M},G)$ such that $\omega(G)=\operatorname{mpf}({\mathcal M},G)$, then $(G,\omega)$ is called a \emph{minimal parametric filling of the type $G$ for the space ${\mathcal M}$}.

\begin {ass}
Let ${\mathcal M}=(M,\rho)$ be a finite metric space. Then
 $$
\operatorname{mf}({\mathcal M})=\inf\bigl\{\operatorname{mpf}({\mathcal M},G)\bigr\},
 $$
where the infimum is taken over all connected graphs $G$ joining $M$.
\end {ass}

It is not difficult to show that to investigate shortest networks and minimal fillings one can restrict the consideration to trees such that all their vertices of degree $1$ and $2$ belong to their boundaries. \textbf{In what follows, we always assume that this condition holds, providing the opposite is not declared}.

To be more precise, we recall the following definition. We say that a tree is a \emph{binary\/} one if the degrees of its vertices can be $1$ or $3$ only, and the boundary consists just of all the vertices of degree $1$. Then each finite metric space has a binary minimal filling (possibly, with some degenerate edges), and a non-degenerate minimal filling (whose type is a tree and all whose vertices of degree $1$ and $2$ belong to its boundary in accordance with the above agreement), see~\cite{ITGromov}.

\subsection{Minimal Realization}\label{sec:realization}
It turns out that the problem on minimal filling can be reduced to Steiner problem in special metric spaces and for special boundaries.

Consider a finite set $M=\{p_1,\ldots,p_n\}$, and let ${\mathcal M}=(M,\rho)$ be a metric space. We put $\rho_{ij}=\rho(p_i,p_j)$. By $\R_\infty^n$  we denote the $n$-dimensional arithmetic space with the norm
 $$
\bigl\|(v^1,\ldots,v^n)\bigr\|_\infty=\max\bigl\{|v^1|,\dots,|v^n|\bigr\},
 $$
and by $\rho_\infty$ the metric on $\R_\infty^n$ generated by $\|\cdot\|_\infty$, i.e., $\rho_\infty(v,w)=\|w-v\|_\infty$. Let us define a mapping $\varphi_{\mathcal M}\colon M\to\R_\infty^n$ as follows:
 $$
\varphi_{\mathcal M}(p_i)={\bar p}_i=(\rho_{i1},\ldots,\rho_{in}).
 $$

\begin{ass}\label{prop:isom_embedding_ellinfty}
The mapping $\varphi_{\mathcal M}$ is an isometry with its image.
\end{ass}

\begin{proof}
This easily follows from the triangle inequality. Indeed,
$$
\bigl\|\bar p_i-\bar p_j\bigr\|=\max_k|\rho_{ik}-\rho_{jk}|\ge\rho_{ij},
$$
because the value $\rho_{ij}$ stands at the $i$th and $j$th places of the vector $\bar p_i-\bar p_j$. On the other hand,
$\rho_{ij}\ge\rho_{ik}-\rho_{jk}$ for any $k$, due to the triangle inequality, hence $\bigl\|\bar p_i-\bar p_j\bigr\|\le \rho_{ij}$, and Assertion is proved.
\end{proof}

The mapping $\varphi_{\mathcal M}$ is called the \emph{Kuratowski isometry}.

Let ${\mathcal G}=(G,\omega)$ be a filling of a space ${\mathcal M}=(M,\rho)$, where $G=(V,E)$, and $d_\omega$ be the pseudo-metric on $V$ generated by the weight function $\omega$. Denote by $E_M$ the edges set of the complete graph on $M$ and put ${\bar G}=(V,{\bar E}=E\cup E_M)$. Let ${\bar\omega}$ be the weight function on ${\bar E}$ coinciding with metric $\rho$ on $E_M$ and with $\omega$ on ${\bar E}\setminus E_M$. Recall that $d_{\bar\omega}$ denotes the pseudo-metric on $V$ generated by ${\bar\omega}$.

We define the network $\Gamma_{\mathcal G}\colon V\to\R_\infty^n$ of the type $G$ as follows:
 $$
\Gamma_{\mathcal G}(v)=\bigl(d_{{\bar\omega}}(v,p_1),\ldots,d_{{\bar\omega}}(v,p_n)\bigr).
 $$
This network is called the \emph{Kuratowski network for the filling ${\mathcal G}$}.

\begin{ass}\label{prop:extend}
We have $\partial\Gamma_{\mathcal G}=\varphi_{\mathcal M}$.
\end{ass}

\begin{proof}
This easily follows from the filling definition. Indeed, the mapping $\partial\Gamma_{\mathcal G}$ is defined on the set $M$ only. By definition,
 $$
\Gamma_{\mathcal G}(p_i)=\bigl(d_{{\bar\omega}}(p_i,p_1),\ldots,d_{{\bar\omega}}(p_i,p_n)\bigr),
 $$
hence it suffices to show that $d_{{\bar\omega}}(p_i,p_k)=\rho_{ik}$ for any $k$. The vertices $p_i$ and $p_k$ are joined by the edge $p_ip_k$ of the weight $\rho_{ik}$ in the graph $\bar G$, and the weight of any other path in $G$ connecting $p_i$ and $p_k$ is more than or equal to $\rho_{ik}$, because $G$ is a filling. Assertion is proved.
\end{proof}

For any network $\Gamma$ in a metric space $(X,d)$ by $\omega_\Gamma$ we denote the \emph{weight function on $G$ induced by the network $\Gamma$}, i.e., $\omega_\Gamma(vw)=d\bigl(\Gamma(v),\Gamma(w)\bigr)$.

\begin{cor}\label{cor:induced_from_ell}
Let ${\mathcal G}=(G,\omega)$ be a minimal parametric filling of a metric space $(M,\rho)$ and $\Gamma=\Gamma_{\mathcal G}$ be the corresponding Kuratowski network. Then $\omega=\omega_\Gamma$.
\end{cor}

Let $\Gamma$ be a network in a metric space ${\mathcal X}$, let $G$ be its parameterizing graph, and ${\mathcal H}=(H,\omega)$ be a weighted graph. We say that \emph{$\Gamma$ and ${\mathcal H}$ are isometric}, if there exists an isomorphism of the weighted graphs ${\mathcal H}$ and ${\mathcal G}=(G,\omega_\Gamma)$.

Corollary~\ref{cor:induced_from_ell} and the existence of minimal parametric and shortest networks in a finite-dimensional normed space~\cite{ITBookWP} imply the following result.

\begin{cor}\label{cor:Kurat_mpf}
Let ${\mathcal M}=(M,\rho)$ be a metric space  consisting of $n$ points, and $\varphi_{\mathcal M}\colon M\to\R_\infty^n$ be the Kuratowski isometry. For any graph $G$ joining $M$ there exists a minimal parametric filling of the type $G$ of the space ${\mathcal M}$. Each minimal parametric filling of the type $G$ of the space ${\mathcal M}$ is isometric to the corresponding Kuratowski network, which is, in this case, a minimal parametric network of the type $G$ with the boundary $\varphi_{\mathcal M}$. Conversely, each minimal parametric network of the type $G$ on $\varphi_{\mathcal M}(M)$ is isometric to some minimal parametric filling of the type $G$ of the space ${\mathcal M}$.
\end{cor}

\begin{cor}\label{cor:Kur_mf}
Let ${\mathcal M}=(M,\rho)$ be a metric space  consisting of $n$ points, and $\varphi_{\mathcal M}\colon M\to\R_\infty^n$ be the Kuratowski isometry. Then there exists a minimal filling ${\mathcal G}$  for ${\mathcal M}$, and the corresponding Kuratowski network $\Gamma_{\mathcal G}$ is a shortest network in the space $\R_\infty^n$ joining the set $\varphi_{\mathcal M}(M)$. Conversely, each shortest network on $\varphi_{\mathcal M}(M)$ is isometric to some minimal filling of the space ${\mathcal M}$.
\end{cor}

\subsection{Minimal Parametric Fillings and Linear Programming} \label{sec:exist}
Let ${\mathcal M}=(M,\rho)$ be a finite metric space connected by a (connected) graph  $G=(V,E)$. As above, by $\Omega({\mathcal M},G)$ we denote the set consisting of all the weight functions  $\omega\colon E\to{\mathbb R}_+$ such that ${\mathcal G}=(G,\omega)$ is a filling of ${\mathcal M}$, and by $\Omega_m({\mathcal M},G)$ we denote its subset consisting of the weight functions such that  ${\mathcal G}$ is a minimal parametric filling of  ${\mathcal M}$.

\begin{ass}\label{prop:opt_weight}
The set $\Omega({\mathcal M},G)$ is closed and convex in the linear space ${\mathbb R}^E$ of all the functions on $E$, and $\Omega_m({\mathcal M},G)\subset\Omega({\mathcal M},G)$ is a nonempty convex compact.
\end{ass}

\begin{proof}
It is easy to see, that the set $\Omega({\mathcal M},G)\subset{\mathbb R}^E$ is determined by the linear inequalities of two types: $\omega(e)\ge 0$, $e\in E$, and $\sum_{e\in\g_{pq}}\omega(e)\ge\rho(p,q)$, where $\g_{pq}$ stands for the unique path in the tree $G$ connecting the boundary vertices $p$ and $q$. Therefore, $\Omega({\mathcal M},G)$ is a convex closed polyhedral subset of $\R^E$ that is equal to the intersection of the corresponding closed half-spaces. The weight functions of minimal parametric fillings correspond to minima points of the linear function $\sum_{e\in E}\om(e)$ restricted to the set $\Omega({\mathcal M},G)$. Thus, the problem of minimal parametric filling finding is a linear programming problem, and the set $\Omega_m({\mathcal M},G)$ of all minima points is a nonempty convex compact polyhedron (the boundedness and, hence, compactness of this set follows from increasing of the objective function with respect to each its variable).
\end{proof}

\subsection{Generalized Fillings}
Investigating the fillings of metric spaces, it turns out to be convenient to expand the class of weighted trees under consideration permitting arbitrary weights of the edges (not only non-negative). The corresponding objects are called \emph{generalized fillings}, \emph{minimal generalized fillings\/} and \emph{minimal parametric ge\-ne\-ra\-lized fillings}. Their weights for a metric space ${\mathcal M}$ and a tree $G$ are denoted by $\operatorname{mf}_-({\mathcal M})$ and $\operatorname{mpf}_-({\mathcal M},G)$, respectively.

For any finite metric space ${\mathcal M}=(M,\rho)$ and a tree $G$ connecting $M$, the next evident inequality is valid: $\operatorname{mpf}_-({\mathcal M},G)\le\operatorname{mpf}({\mathcal M},G)$. And it is not difficult to construct an example, when this inequality becomes strict, see~Figure~\ref{fig:mf-minus}. However, for minimal generalized fillings the following result holds, see~\cite{IOST}.

\begin{thm}[Ivanov, Ovsyannikov, Strelkova, Tuzhilin]\label{th:IOST}
For an arbitrary finite metric space ${\mathcal M}$, the set of all its minimal generalized fillings contains its minimal filling, i.e\. a generalized minimal filling with nonnegative weight function. Hence,
$\operatorname{mf}_-({\mathcal M})=\operatorname{mf}({\mathcal M})$.
\end{thm}

\Pic{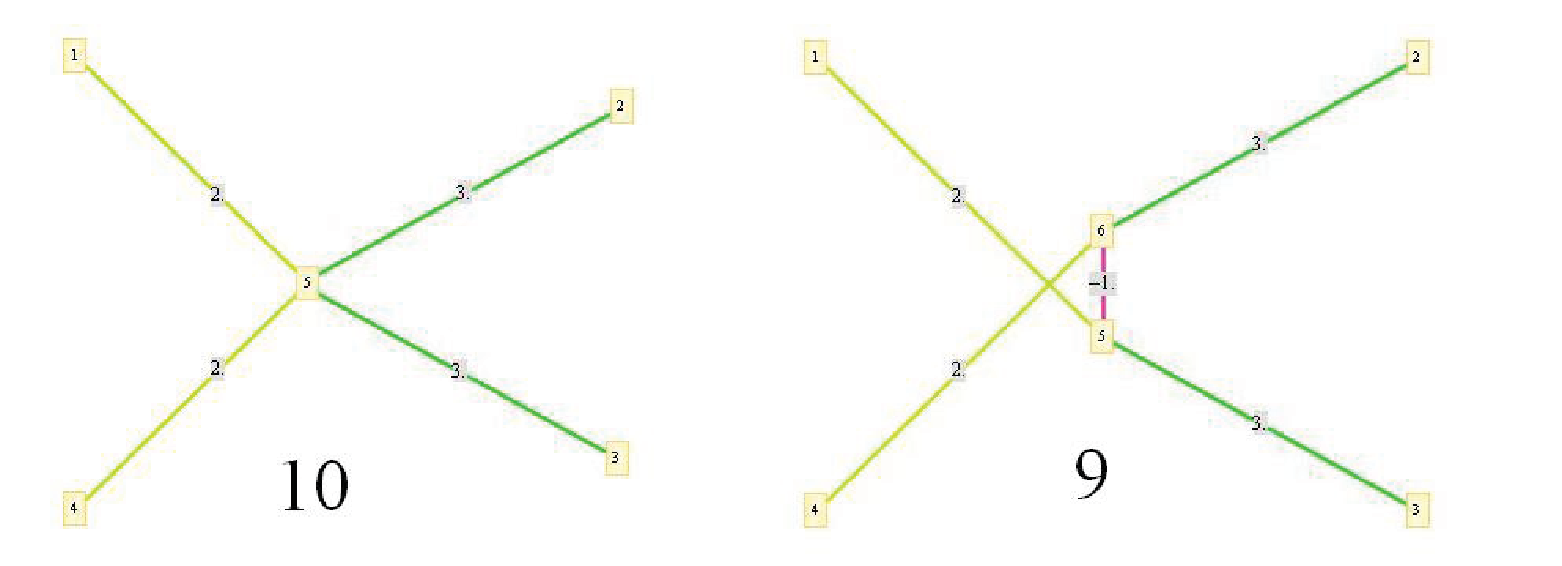}{Minimal parametric filling (left) and minimal generalized parametric filling (right) of the vertex set of the plane rectangle with sides $3$ and $4$. The type is the same: the moustaches connects the diagonal pairs of the vertices. The interior edge has to be zero in the case of the filling and can be negative in the case of the generalized filling.  Here $9=\operatorname{mpf}_-({\mathcal M},G)<\operatorname{mpf}({\mathcal M},G)=10$.}{fig:mf-minus}{144}

\subsection{Formula for the Weight of Minimal Filling}
Let ${\mathcal M}=(M,\rho)$ be a finite metric space, and $G$ be a tree connecting $M$. Choose an arbitrary embedding $G'$ of the tree  $G$ into the plane. Consider a walk around the tree $G'$. We draw the points of $M$ consecutive with respect to this walk as a consecutive points of the circle $S^1$. Notice that each vertex $p$ from $M$ appears $\deg p$ times.  For each vertex $p\in M$ of degree more than $1$, we choose just one arbitrary point from the corresponding points of the circle. So, we construct an injection $\nu\colon M\to S^1$. Define a cyclic permutation $\pi$ as follows:  $\pi(p)=q$, where $\nu(q)$ follows after $\nu(p)$ on the circle $S^1$. We say that $\pi$ \emph{is generated by the embedding $G'$} (this procedure is not unique due to different possible choices of $\nu$). Each $\pi$ generated in this manner is called a \emph{tour of $M$ with respect to $G$}. The set of all tours on $M$ with respect to $G$ is denoted by ${\mathcal O}(M,G)$. For each tour $\pi\in{\mathcal O}(M,G)$ we put
 $$
p({\mathcal M},G,\pi)=\frac1{2}\sum_{x\in M}\rho\bigl(x,\pi(x)\bigr)
 $$
and we call this value by the \emph{half-perimeter of the space ${\mathcal M}$ with respect to the tour $\pi$}. The minimal value of $p({\mathcal M},G,\pi)$ over all $\pi\in{\mathcal O}(M,G)$ for all possible $G$ (in fact, over all possible cyclic permutations $\pi$ on $M$) is called the \emph{half-perimeter of the space ${\mathcal M}$}.

A.~Ivanov and A.~Tuzhilin proposed the following hypothesis.

\begin{conj}\label{conj:min-fill-formula}
For an arbitrary metric space ${\mathcal M}=(M,\rho)$ the following formula is valid
 $$
\operatorname{mf}({\mathcal M})=\min_G\max_{\pi\in{\mathcal O}(M,G)}p({\mathcal M},G,\pi),
 $$
where minimum is taken over all binary trees $G$ connecting $M$.
\end{conj}

A.\,Yu.~Eremin~\cite{Eremin} constructed a counter-example to the Conjecture~\ref{conj:min-fill-formula} and showed that if one changes the concept of tour by the one of multitour, introduced by him, then the Conjecture~\ref{conj:min-fill-formula} holds.

To define the multitours, let us consider the graph in which every edge of $G$ is taken with the multiplicity $2k$, $k\ge1$. The resulting graph possesses an Euler cycle consisting of \emph{irreducible\/} boundary paths --- the ones which do not contain properly other boundary paths. This Euler cycle generates a bijection $\pi\colon X\to X$, where $X=\sqcup_{i=1}^kM$, which is called \emph{multitour of $M$ with respect to $G$}, see an example in Figure~\ref{fig:multour}. The set of all multitours on $M$ with respect to $G$ is denoted by ${\mathcal O}_\mu(M,G)$.

\Pic{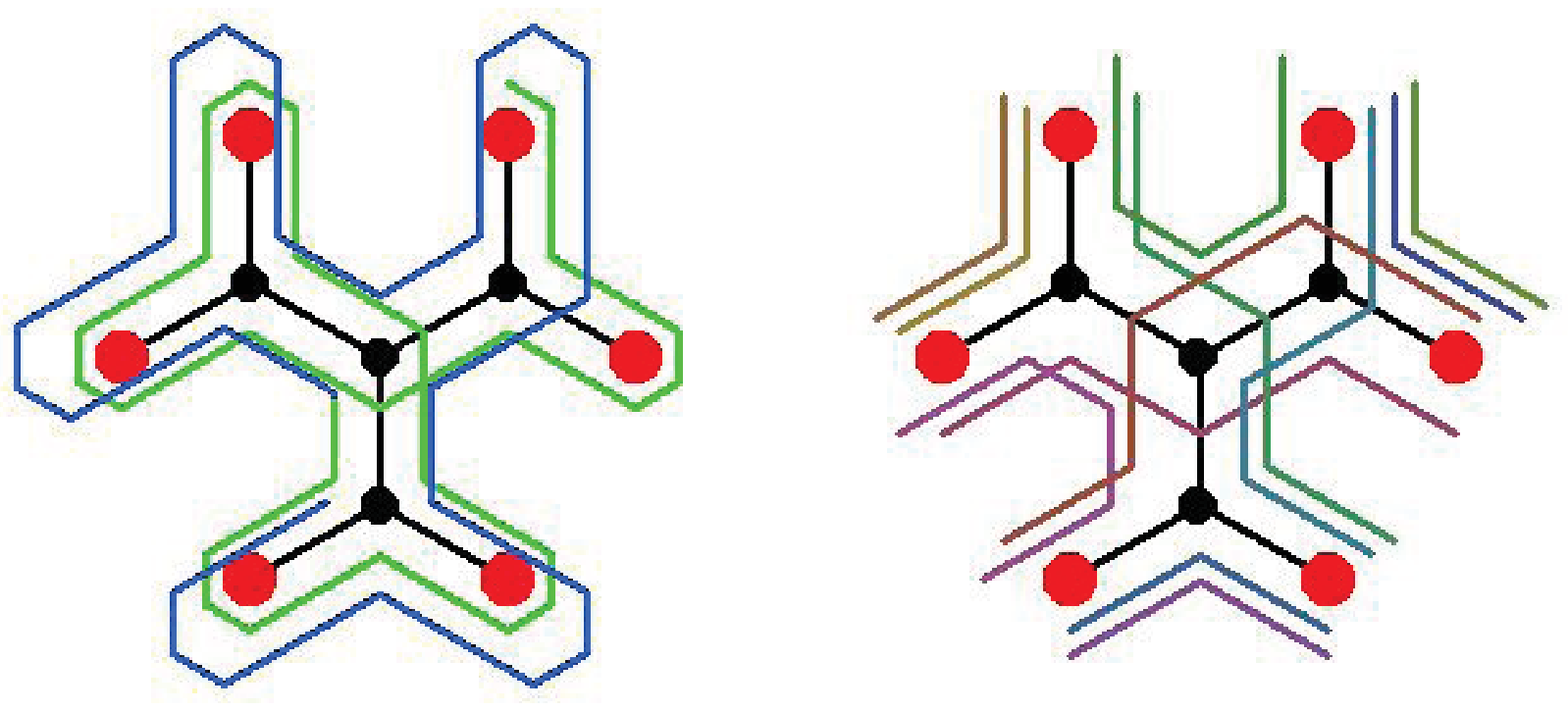}{A part of a moultitour with multiplicity $2$ (left), and the irreducible boundary paths the multitour consists from (right). The multitour starts as a green polygonal line and becomes blue when multiplicity of edges becomes more than $2$.}{fig:multour}{144}

Let ${\mathcal M}=(M,\rho)$ be a finite metric space, and $G$ be a tree connecting $M$. As in the case of tours, for each multitour $\pi\in{\mathcal O}_\mu(M,G)$ we put
 $$
p({\mathcal M},G,\pi)=\frac1{2k}\sum_{x\in X}\rho\bigl(x,\pi(x)\bigr).
 $$

\begin{thm}[A.\,Yu.~Eremin]\label{th:eremin}
For an arbitrary finite metric space ${\mathcal M}=(M,\rho)$ and an arbitrary tree $G$ joining $M$, the weight of minimal parametric generalized filling can be calculated as follows
 $$
\operatorname{mpf}_-({\mathcal M},G)=\max\bigl\{p({\mathcal M},G,\pi)\mid \pi\in{\mathcal O}_\mu(M,G)\bigr\}.
 $$
The weight of minimal filling can be calculated as follows
 $$
\operatorname{mf}({\mathcal M})=\operatorname{mf}_-({\mathcal M})=\min_G\max\bigl\{p({\mathcal M},G,\pi)\mid \pi\in{\mathcal O}_\mu(M,G)\bigr\},
 $$
where minimum is taken over all binary trees $G$ connecting $M$.
\end{thm}

\subsection{Minimal Fillings for Generic Metric Spaces}
Theorem~\ref{th:eremin} gives an opportunity to get several interesting corollaries. To formulate one of them, we need to define what is a ``generic'' metric space. Notice that the set of all metric spaces consisting of $n$ points can be naturally identified with a convex cone in ${\mathbb R}^{n(n-1)/2}$ (it suffices to enumerate the set of all two-elements subsets of these spaces and assign to each such space the vector of the distances between the pairs of points). This representation gives us an opportunity to speak about topological properties of families of metric spaces consisting of a fixed number of points.

We say, that some property holds for a {\em generic metric space}, if for any $n$ this property is valid for an everywhere dense set of $n$-point metric spaces.

The following result can be found in~\cite{Eremin}.

\begin{cor}[A.\,Yu.~Eremin]
Each general finite metric space has a minimal filling which is a nondegenerate binary tree.
\end{cor}

\subsection{Additive Spaces and Minimal Fillings} \label{sec:additive}
The additive spaces are very popular in bioinformatics, playing an important role in evolution theory and, more general, in an hierarchy modeling. Recall that a finite metric space ${\mathcal M}=(M,\rho)$ is called \emph{additive}, if $M$ can be joined by a weighted tree ${\mathcal G}=(G,\omega)$ such that $\rho$ coincides with the restriction of $d_\omega$ onto $M$. The tree ${\mathcal G}$ in this case is called a \emph{generating tree\/} for the space ${\mathcal M}$.

Not any metric space is additive. An additivity criterion can be stated in terms of so-called \emph{$4$ points rule}: for any four points $p_i$, $p_j$, $p_k$, $p_l$, the values $\rho(p_i,p_j)+\rho(p_k,p_l)$, $\rho(p_i,p_k)+\rho(p_j,p_l)$, $\rho(p_i,p_l)+\rho(p_j,p_k)$ are the lengths of sides of an isosceles triangle whose base does not exceed its other sides.

\begin{thm}[\cite{Zaretskij}, \cite{SimoesPereira}, \cite{Smolenskij}, \cite{HakimiYau}]\label{prop:add-unique}
A metric space is additive, if and only if it satisfies the $4$ points rule. In the class of non-degenerate weighted trees, the generating tree of an additive metric space is unique.
\end{thm}

The next criterion solves completely the minimal filling problem for additive metric spaces.

\begin{thm}\label{th:additive=minimum}
Minimal fillings of an additive metric space are exactly its generating trees.
\end{thm}

The next additivity criterion is obtained by O.\,V.~Rubleva, a student of Mechanical and Mathematical Faculty of Moscow State University, see~\cite{Rubleva}.

\begin{ass}[O.\,V.~Rubleva]\label{prop:Rubleva}
The weight of a minimal filling of a finite metric space is equal to the half-perimeter of this space, if and only if this space is additive.
\end{ass}

In the scope of Assertion~\ref{prop:Rubleva}, we conjectured that if there exists a tree connecting a metric space such that all the corresponding half-perimeters are equal to each other, then the space is additive. It turns out that it is not true. Z.\,N.~Ovsyannikov suggested to consider a wider class of spaces, so called pseudo-additive spaces, for which our conjecture becomes true, see~\cite{Ovs}.

A finite metric space ${\mathcal M}=(M,\rho)$ is said to be \emph{pseudo-additive}, if the metric $\rho$ coincides with $d_\omega$ for a generalized weighted tree $(G,\omega)$ (which is also called \emph{generating\/}), where the weight function $\omega$ can take arbitrary (not necessary nonnegative) values. Z.\,N.~Ovsyannikov shows that these spaces can be described in terms of so-called \emph{weak $4$-points rule}: for any four points $p_i$, $p_j$, $p_k$, $p_l$, the values $\rho(p_i,p_j)+\rho(p_k,p_l)$, $\rho(p_i,p_k)+\rho(p_j,p_l)$, $\rho(p_i,p_l)+\rho(p_j,p_k)$ are the lengths of sides of an isosceles triangle. The generating tree is also unique in the class of non-degenerate trees.  Moreover, the following result is valid, see~\cite{Ovs}.

\begin{thm}[Z.\,N.~Ovsyannikov]\label{th:ovs}
Let ${\mathcal M}=(M,\rho)$ be a finite metric space. Then the following statements are equivalent.
\begin{itemize}
 \item There exist a tree $G$ such that $M$ coincides with the set of degree $1$ vertices of $G$ and all the half-perimeters $p(M,G,\pi)$ of $M$ corresponding to the tours around $G$ are equal to each other.
 \item The space ${\mathcal M}$ is pseudo-additive.
\end{itemize}
Moreover, the three $G$ in this case is a generating tree for the space ${\mathcal M}$.
\end{thm}

It would be interesting to see, what role these pseudo-additive spaces could play in applications.

\subsection{Examples of Minimal Fillings} \label{sec:examp}
Now let us give several examples of minimal filling and demonstrate how to use the technique elaborated above.

\subsubsection{Triangle}\label{subsec:triangle}
Let ${\mathcal M}=(M,\rho)$ consist of three points $p_1$, $p_2$, and $p_3$. Put $\rho_{ij}=\rho(p_i,p_j)$. Consider the tree $G=(V,E)$ with $V=M\cup\{v\}$ and $E=\{vp_i\}_{i=1}^3$. Define the weight function $\omega$ on $E$ by the following formula:
 $$
\omega(e_i)=\dfrac{\rho_{ij}+\rho_{ik}-\rho_{jk}}2,
 $$
where $\{i,j,k\}=\{1,2,3\}$. Notice that $d_\omega$ restricted onto $M$ coincides with $\rho$. Therefore, ${\mathcal M}$ is an additive space, ${\mathcal G}=(G,\omega)$ is a generating tree for ${\mathcal M}$, and, due to Theorem~\ref{th:additive=minimum}, ${\mathcal G}$ is a minimal filling of ${\mathcal M}$.

Recall that the value $(\rho_{ij}+\rho_{ik}-\rho_{jk})/2$ is called by the {\em Gromov product\/} $(p_j,p_k)_{p_i}$ of the points $p_j$ and $p_k$ of the space ${\mathcal M}$ with respect to the point $p_i$, see~\cite{GromHypGr}.

\subsubsection{Regular Simplex}\label{subsec:simplex}
Let all the distances in the metric space ${\mathcal M}$ are the same and are equal to $d$, i.e.  ${\mathcal M}$ is a regular simplex. Then the weighted tree ${\mathcal G}=(G,\omega)$, $G=(V,E)$, with the vertex set $V=M\cup\{v\}$ and edges $vm$, $m\in M$, of the weight $d/2$ is generating for ${\mathcal M}$. Therefore, the space ${\mathcal M}$ is additive, and, due to Theorem~\ref{th:additive=minimum}, ${\mathcal G}$ is its unique nondegenerate minimal filling. If $n$ is the number of points in $M$, then the weight of the minimal filling is equal to $dn/2$.

\subsubsection{Star}
If a minimal filling ${\mathcal G}=(G,\omega)$ of a space ${\mathcal M}=(M,\rho)$ is a star whose single interior vertex $v$ is joined with each point $p_i\in M$, $1\le i\le n$, $n\ge 3$, then the metric space ${\mathcal M}$ is additive~\cite{ITGromov}. In this case the weights of edges can be calculated easily. Indeed, put $e_i=vp_i$. Since a subspace of an additive space is additive itself, then we can use the results for three-points additive space, see above. So, we have $\omega(e_i)=(p_j,p_k)_{p_i}$, where $p_i$, $p_j$, and $p_k$ are arbitrary distinct  boundary vertices.

\subsubsection{Mustaches of Degree more than 2}
Let $G=(V,E)$ be an arbitrary tree, and $v\in V$ be an interior vertex of degree $(k+1)\ge3$ adjacent with $k$ vertices $w_1,\ldots,w_k$ from $\partial G$. Then the set of the vertices $\{w_1,\ldots,w_k\}$, and also the set of the edges  $\{vw_1,\ldots,vw_k\}$, are referred as {\em mustaches}. The number $k$ is called by the {\em degree}, and the vertex $v$ is called by the {\em common vertex of the mustaches}. An edge incident to $v$ and not belonging to $\{vw_1,\ldots,vw_k\}$ is called the {\em root edge\/} of the mustaches under consideration.

As it is shown in~\cite{ITGromov}, any mustaches of a minimal filling of a metric space forms an additive subspace. If the degree of such mustaches is more than $2$, then we can calculate the weights of all the edges containing in the mustaches just in the same way as in the case of a star.

\subsubsection{Four-Points Spaces}
Here we give a complete description of minimal fillings for four-points spaces, see details in~\cite{ITGromov}.

\begin{prop}
Let $M=\{p_1,p_2,p_3,p_4\}$, and $\rho$ be an arbitrary metric on $M$. Put $\rho_{ij}=\rho(p_i,p_j)$. Then the weight of a minimal filling ${\mathcal G}=(G,\omega)$ of the space ${\mathcal M}=(M,\rho)$ is given by the following formula
 $$
\frac12\bigl(\min\{\rho_{12}+\rho_{34},\,\rho_{13}+\rho_{24},\,\rho_{14}+\rho_{23}\}+
\max\{\rho_{12}+\rho_{34},\,\rho_{13}+\rho_{24},\,\rho_{14}+\rho_{23}\}\bigr).
 $$
If the minimum in this formula is equal to $\rho_{ij}+\rho_{rs}$, then the type of minimal filling is the binary tree with the mustaches $\{p_i,p_j\}$ and $\{p_r,p_s\}$.
\end{prop}

We apply the obtained result to the vertex set of a planar convex quadrangle.

\begin{cor}\label{cor:conv4gon}
Let $M$ be the vertex set of a convex quadrangle $p_1p_2p_3p_4\subset{\mathbb R}^2$ and $\rho(p_i,p_j)=\|p_i-p_j\|$. The weight of a minimal filling of the space $(M,\rho)$ is equal to
 $$
\frac12\min\bigl(\rho_{12}+\rho_{34},\,\rho_{14}+\rho_{23}\bigr)+
\frac{\rho_{13}+\rho_{24}}2.
 $$
The topology of minimal filling is a binary tree with mustaches corresponding to opposite sides of the less total length.
\end{cor}

\subsection{Ratios} \label{sec:ratios}
The Steiner ratio is discussed in Section~\ref{sec:sr}. Here we define some other ratios based on minimal fillings, which could be more available for calculating, and which could be useful to calculate the Steiner ratio, as we hope.

\subsection{Steiner--Gromov Ratio}
For convenience, the sets consisting of more than a single point are referred as \emph{nontrivial}. Let ${\mathcal X}=(X,\rho)$ be an arbitrary metric space, and let $M\subset X$ be some finite subset. Recall that by $\operatorname{mst}(M,\rho)$ we denote the length of minimal spanning tree of the space $(M,\rho)$. Further, for nontrivial $M$, we define the value
 $$
\operatorname{sgr}(M)=\operatorname{mf}(M,\rho)/\operatorname{mst}(M,\rho)
 $$
and call it the \emph{Steiner--Gromov ratio\/} of the subset $M$. The value $\inf\operatorname{sgr}(M)$, where the infimum is taken over all nontrivial finite subsets of ${\mathcal X}$, consisting of at most $n$ vertices is denoted by $\operatorname{sgr}_n({\mathcal X})$ and is called the \emph{degree $n$ Steiner--Gromov ratio of the space ${\mathcal X}$}. At last, the value $\inf\operatorname{sgr}_n({\mathcal X})$, where the infimum is taken over all positive integers $n>1$ is called the \emph{Steiner--Gromov ratio of the space ${\mathcal X}$} and is denoted by $\operatorname{sgr}({\mathcal X})$, or by $\operatorname{sgr}(X)$, if it is clear what particular metric on $X$ is considered. Notice that $\operatorname{sgr}_n({\mathcal X})$ is a non-increasing function on $n$. Besides, it is easy to see that $\operatorname{sgr}_2({\mathcal X})=1$ and $\operatorname{sgr}_3({\mathcal X})=3/4$ for any nontrivial metric space ${\mathcal X}$.

\begin{ass}\label{prop:steiner_ratio}
The Steiner--Gromov ratio of an arbitrary metric space is not less than $1/2$. There exist metric spaces whose Steiner--Gromov ratio equals to $1/2$.
\end{ass}

Recently, A.~Pakhomova, a student of Mechanical and Mathematical Faculty of Moscow State University, obtained an exact general estimate for the degree $n$ Steiner--Gromov ratio, see~\cite{Pakh}.

\begin{ass}[A.~Pakhomova]
For any metric space ${\mathcal X}$ the estimate
$$
\operatorname{sgr}_n({\mathcal X})\ge\frac{n}{2n-2}
$$
is valid. Moreover, this estimate is exact, i.e. for any $n\ge 3$ there exists a metric space ${\mathcal X}_n$ such that $\operatorname{sgr}_n({\mathcal X}_n)=n/(2n-2)$.
\end{ass}

Also recently, Z.~Ovsyannikov~\cite{Ovs-h} investigated the metric space of all compact subsets of Euclidean plane endowed with Hausdorff metric.

\begin{ass}[Z.~Ovsyannikov]
The Steiner ratio and the Steiner--Gromov ratio of the metric space of all compact subsets of Euclidean plane endowed with Hausdorff metric are equal to $1/2$.
\end{ass}

\subsection{Steiner Subratio}
Let ${\mathcal X}=(X,\rho)$ be an arbitrary metric space, and let $M\subset{\mathcal X}$ be some its finite subset. Recall that by $\operatorname{smt}(M,\rho)$ we denote the length of Steiner minimal tree joining  $M$. Further, for nontrivial subsets $M$, we define the value
 $$
\operatorname{ssr}(M)=\operatorname{mf}(M,\rho)/\operatorname{smt}(M,\rho)
 $$
and call it by the \emph{Steiner subratio\/} of the set $M$. The value $\inf\operatorname{ssr}(M)$, where infimum is taken over all nontrivial finite subsets of ${\mathcal X}$ consisting of at most $n>1$ points, is denoted by $\operatorname{ssr}_n({\mathcal X})$ and is called the \emph{degree $n$ Steiner subratio of the space ${\mathcal X}$}. At last, the value $\inf\operatorname{ssr}_n({\mathcal X})$, where the infimum is taken over all positive integers  $n>1$, is called the \emph{Steiner subratio of the space ${\mathcal X}$} and is denoted by $\operatorname{ssr}({\mathcal X})$, or by $\operatorname{ssr}(X)$, if it is clear what particular metric  on $X$ is considered. Notice that $\operatorname{ssr}_n({\mathcal X})$ is a nonincreasing function on $n$.  Besides, it is easy to see that $\operatorname{ssr}_2({\mathcal X})=1$ for any nontrivial metric space ${\mathcal X}$.

\begin{prop}
$\operatorname{ssr}_3({\mathbb R}^n)=\sqrt3/2$.
\end{prop}

The next result is obtained by E.\,I.~Filonenko, a student of Mechanical and Mathematical Department of Moscow State University, see~\cite{Filonenko}.

\begin{prop}[E.\,I.~Filonenko]
$\operatorname{ssr}_4({\mathbb R}^2)=\sqrt3/2$.
\end{prop}

\begin{conj}\label{conj:subrat}
The Steiner subratio of the Euclidean plane is achieved at the regular triangle and, hence, is equal to $\sqrt3/2$.
\end{conj}

Recently, A.~Pakhomova obtained an exact general estimate foe the degree $n$ Steiner subratio, see~\cite{Pakh}.

\begin{prop}[A.~Pakhomova]
For any metric space $\{\mathcal X\}$ the estimate
$$
\operatorname{ssr}_n({\mathcal X})\ge\frac{n}{2n-2}
$$
is valid. Moreover, this estimate is exact, i.e. for any $n\ge 3$ there exists a metric space ${\mathcal X}_n$ such that $\operatorname{ssr}_n({\mathcal X}_n)=n/(2n-2)$.
\end{prop}

Also recently, Z.~Ovsyannikov~\cite{Ovs-h} investigated the metric space of all compact subsets of Euclidean plane endowed with Hausdorff metric.

\begin{prop}[Z.~Ovsyannikov]
Let ${\mathcal C}$ be the metric space of all compact subsets of Euclidean plane endowed with Hausdorff metric. Then $\operatorname{ssr}_3({\mathcal C})=3/4$ and  $\operatorname{ssr}_4({\mathcal C})=2/3$.
\end{prop}

\medskip

\end{document}